\documentclass[11pt]{article}
\usepackage{amsmath,amssymb,amsthm, mathtools}
\usepackage{enumerate}
\usepackage{xfrac}
\usepackage[symbol,perpage]{footmisc}
\usepackage{hyperref}
\usepackage{enumitem}
\usepackage{graphicx, subcaption, xcolor}
\usepackage[margin=0.9in]{geometry}
\usepackage{authblk}
\usepackage{tikz-cd}
\theoremstyle{plain}
\newtheorem{thm}{Theorem}[section]
\newtheorem{prop}[thm]{Proposition}
\newtheorem{lemma}[thm]{Lemma}

\newtheorem{claim}[thm]{Claim}
\newtheorem*{claim*}{Claim}

\theoremstyle{definition}
\newtheorem{defi}[thm]{Definition}

\newtheorem{obs}[thm]{Remark}
\newtheorem{ex}[thm]{Example}

\newtheorem*{defi*}{Definition}
\usepackage{chngcntr}
\counterwithin{figure}{section}

\usepackage[english]{babel}

\theoremstyle{plain}
\newtheorem*{namedthm}{\namedthmname}
\newcounter{namedthm}
\makeatletter
\newenvironment{named}[1]
{\def\namedthmname{#1}%
	\refstepcounter{namedthm}%
	\namedthm\def\@currentlabel{#1}}
{\endnamedthm}
\makeatother

\def\dist{\mathrm{dist}}

\emergencystretch=\maxdimen
\title{Boundaries of hyperbolic and simply parabolic Baker domains}
\author[1]{Anna Jov\'e\thanks{This work is supported by the Spanish government grants FPI PRE2021-097372 and PID2023-147252NB-I00. Corresponding author. \url{ajovecam7@alumnes.ub.edu}.}}
\affil[1]{\small Departament de Matemàtiques i Informàtica, Universitat de Barcelona, Barcelona, Spain}
\date{\today}

\begin{document}
\maketitle
\begin{abstract}
	We study the boundaries of non-univalent simply connected Baker domains of transcendental maps (both entire and meromorphic), of hyperbolic and simply parabolic type. We prove non-ergodicity and non-recurrence for the boundary map, and additional properties concerning the Julia set and the set of singularities of the associated inner function, and the topology and the dynamics on the boundary of the Baker domain. In particular, we prove the existence of points on the boundary whose orbit does not converge to  infinity through the dynamical access, in the sense of Carathéodory. Finally, under mild conditions on the postsingular set, we prove the existence of periodic points on the boundary of such Baker domains.
\end{abstract}

\section{Introduction}

Consider the discrete  dynamical system generated by  a transcendental meromorphic map $ f \colon\mathbb{C}\to\widehat{\mathbb{C}} $, i.e.   the sequences of iterates $ \left\lbrace f^n(z)\right\rbrace _n $, where $ z\in{\mathbb{C}} $.  These dynamical systems arise naturally, for example, from the popular Newton’s root-finding method applied to entire functions. For general background on the iteration of meromorphic maps, we refer to \cite{bergweiler93}. Since the class of meromorphic maps is not closed under composition, and hence unsuitable to deal with iteration, 	we shall work with functions in class $ \mathbb{K} $, which are meromorphic outside a compact countable set of singularities, and for which the general theory of iteration of meromorphic maps extends successfully  (see e.g. \cite{Bolsch-thesis, BakerDominguezHerring}). Formally, $ f\in\mathbb{K} $ if there exists a compact countable set $ E(f)\subset \widehat{\mathbb{C}} $ such that \[f\colon\widehat{\mathbb{C}}\smallsetminus E(f)\to \widehat{\mathbb{C}} \] is meromorphic in $ \Omega\coloneqq  \widehat{\mathbb{C}}\smallsetminus E(f)$ but in no larger set. 

In this situation, the Riemann sphere $ \widehat{\mathbb{C}} $, regarded as the phase space of the dynamical system, is divided into two totally invariant sets: the {\em Fatou set} $ \mathcal{F}(f) $,  the set of points $ z\in\mathbb{C} $ such that $ \left\lbrace f^n\right\rbrace _{n\in\mathbb{N}} $ is well-defined and forms a normal family in some neighbourhood of $ z$; and the {\em Julia set} $ \mathcal{J}(f) $, its complement, where the dynamics is chaotic. By definition, the set of  singularities $ E(f) $ belongs to the Julia set.

The Fatou set is open and consists in general of infinitely many components,  called \textit{Fatou components}. Due to the invariance of the Fatou and Julia sets, Fatou components are either periodic, preperiodic or wandering. Periodic Fatou components, which can be assumed  to be invariant, are classified into \textit{rotation domains} (\textit{Siegel disks} or \textit{Herman rings}), \textit{attracting} and \textit{parabolic basins}, and \textit{Baker domains}, being the latter exclusive of transcendental functions. Indeed, a Baker domain is, by definition, an invariant  Fatou component in which iterates converge to an essential singularity of $ f $ uniformly on compact sets. Conjugating $ f $ by a Möbius transformation if needed, we shall assume that this essential singularity is placed at $ \infty $.

A standard approach to study invariant simply connected Fatou components is to conjugate the dynamics in the Fatou component $ U $ to the dynamics of an inner function of the unit disk $ \mathbb{D} $. Indeed, consider a Riemann map $ \varphi\colon\mathbb{D}\to U $, and the function \[g\colon\mathbb{D}\to\mathbb{D}, \hspace{0.5cm} g\coloneqq \varphi^{-1}\circ f\circ \varphi.\]The map $ g $ is the {\em inner function associated with $ f|_U $}, and it is unique up to conjugacy with automorphisms of $ \mathbb{D} $.

The dynamics of $ g|_{\mathbb{D}} $, and hence of $ f|_U $, follows essentially from the Denjoy-Wolff Theorem, which asserts the existence of a distinguished point $ p\in\overline{\mathbb{D}} $ towards which all points converge under iteration, except in the case when $ g $ is conjugate to a rotation (the case of Siegel disks). Note that $ p\in\mathbb{D} $ if $ U $ is an attracting basin (and $ g $ is called {\em elliptic}), and $ p\in\partial \mathbb{D} $ if $ U $ is a parabolic basin or a Baker domain.  

Holomorphic self-maps of $ \mathbb{D} $, $ g\colon\mathbb{D}\to\mathbb{D} $,
with Denjoy-Wolff point $ p\in \partial\mathbb{D} $ are classified into \textit{doubly parabolic}, \textit{hyperbolic} and \textit{simply parabolic}, according to the dynamics near the Denjoy-Wolff point,  following a celebrated result by Cowen \cite{cowen}. In the particular case when the map extends analytically to a neighbourhood of $ p $ (and thus $ p $ is fixed by $ g $), the previous classification depends on the character of $ p $ as a fixed point of $ g $. More preciely,  $ g $ is {\em hyperbolic} if if $ p $ is an attracting fixed point; $ g $ is {\em simply parabolic} if parabolic fixed point with a single petal;
 and $ g $ is {\em doubly parabolic} if $ p $ is a parabolic fixed point with two petals. One should keep in mind that, however, that Cowen's result does not require the holomorphic extension around the Denjoy-Wolff point, and it is stated in terms of internal dynamics in a one-sided neighbourhood of the Denjoy-Wolff point. For more details, see Section \ref{sect-inner}, and specifically, Theorem \ref{teo-cowen}. 
 
 \begin{figure}[htb!]\centering
	\captionsetup[subfigure]{labelformat=empty}
	\hfill
 	\begin{subfigure}[b]{0.27\textwidth}
 		\includegraphics[width=\textwidth]{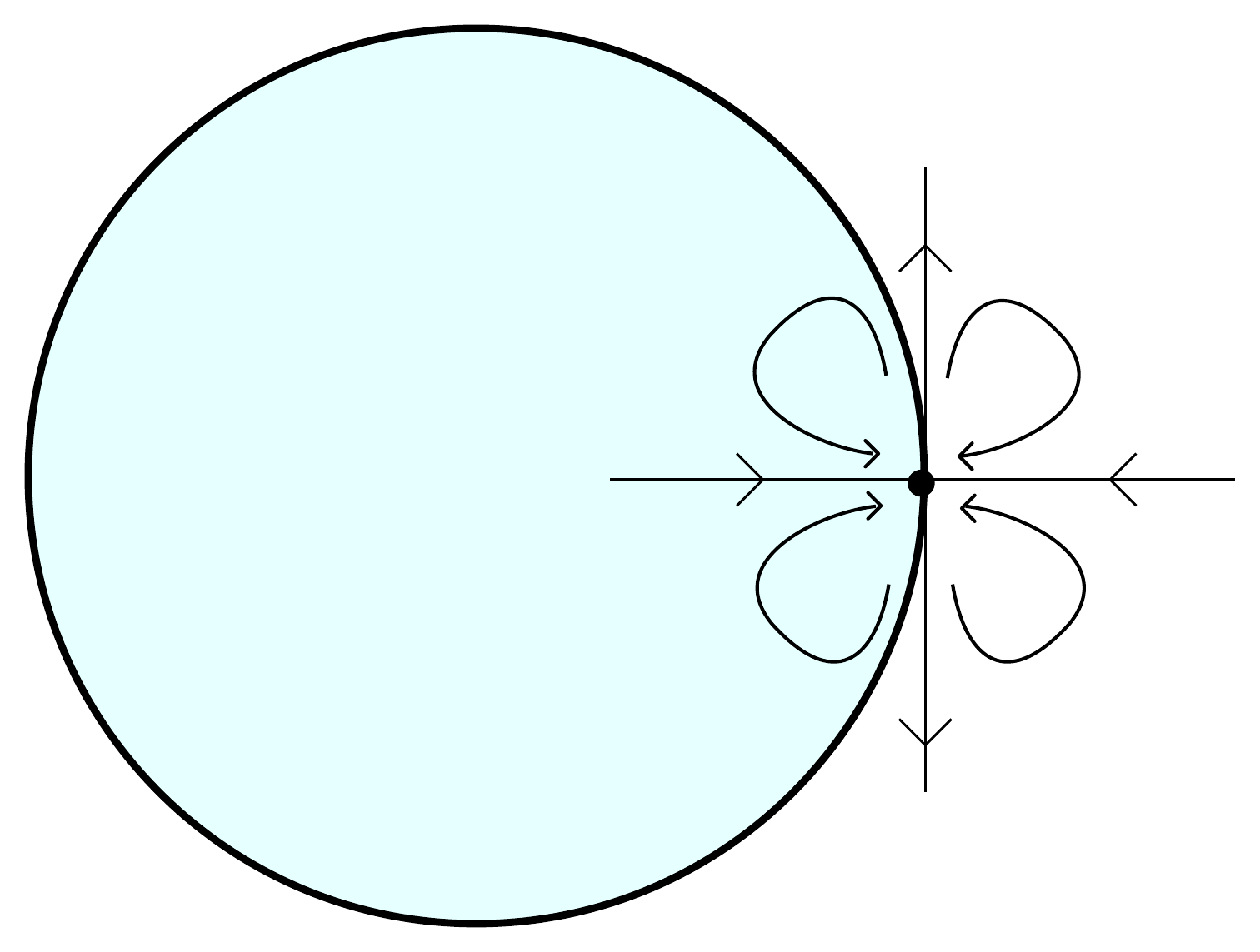}
 		\caption{\footnotesize Doubly parabolic}
 		
 	\end{subfigure}
  	 \hfill
  	 \hfill
 	\begin{subfigure}[b]{0.27\textwidth}
 		\includegraphics[width=\textwidth, ]{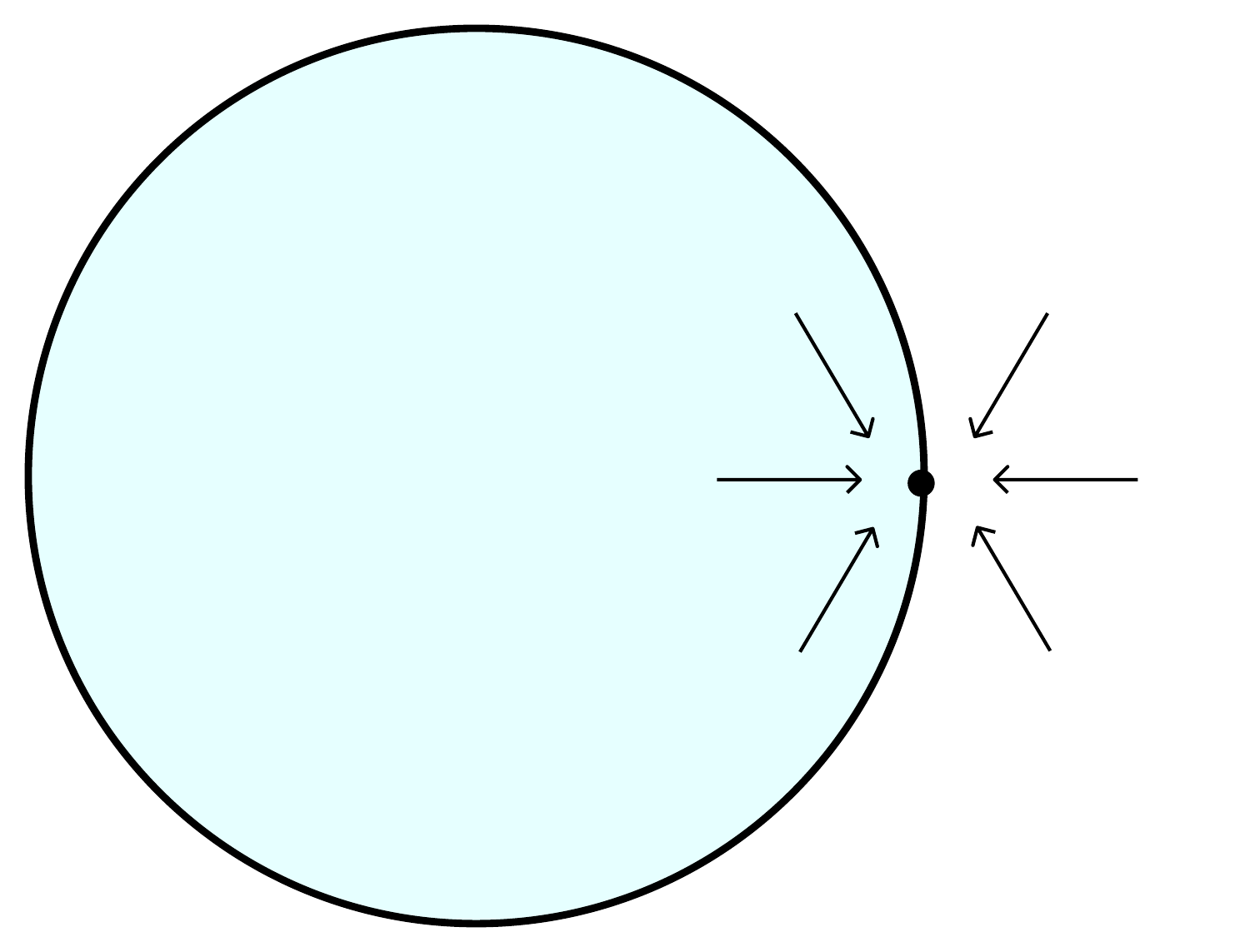}
 		\caption{\footnotesize Hyperbolic}
 		
 	\end{subfigure}
 \hfill
 \hfill
 	\begin{subfigure}[b]{0.27\textwidth}
 		\includegraphics[width=\textwidth]{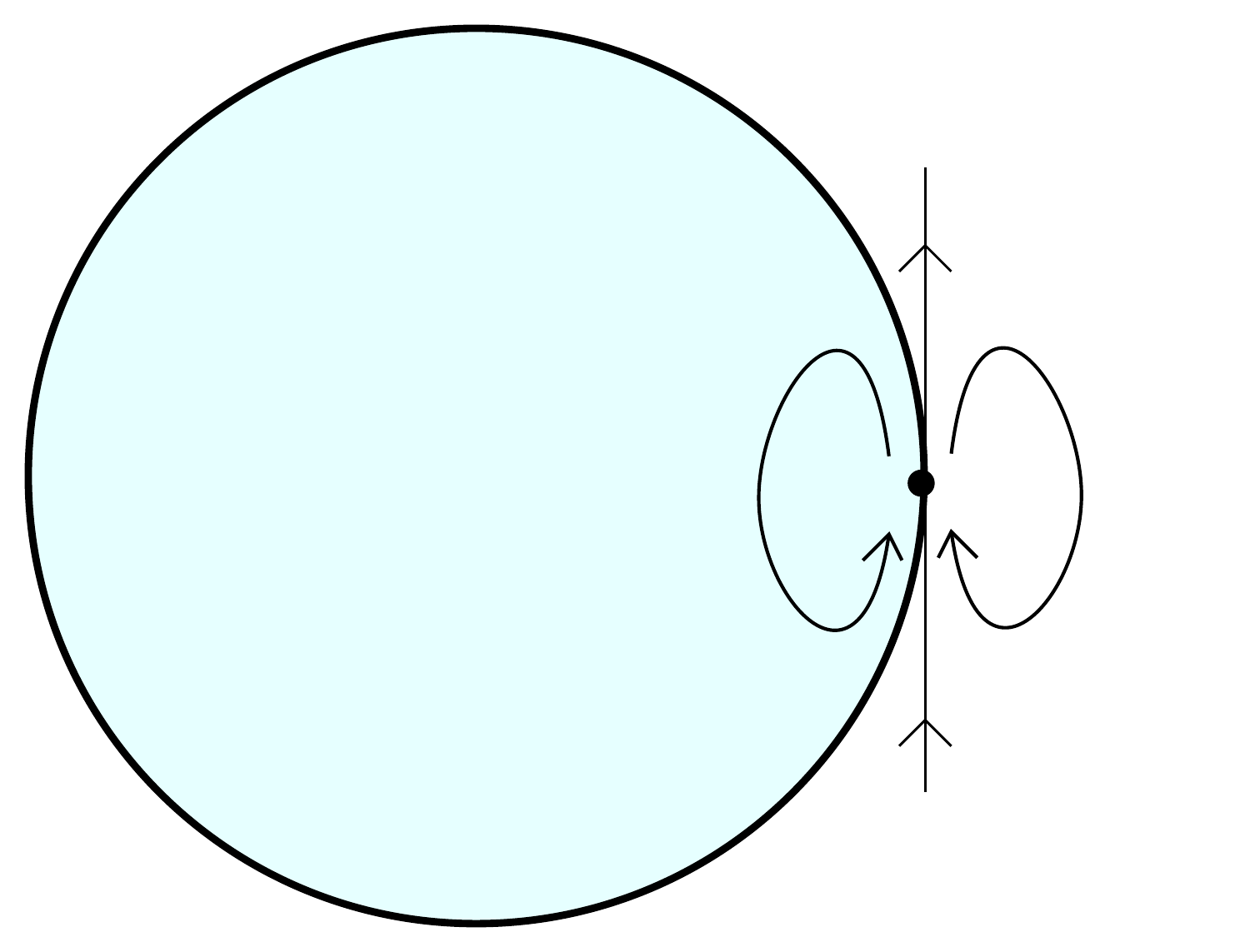}
 		\caption{\footnotesize Simply parabolic}
 		
 	\end{subfigure}
\hfill
 	\caption{{\footnotesize Possible dynamics around the Denjoy-Wolff point, when it is not a singularity.}}
 \end{figure}

Inner functions associated with parabolic basins are always of doubly parabolic type, while Baker domains can exhibit any of the three possible behaviours \cite{Konig,FH}, thus  establishing a classification among them.

The associated inner function is also a powerful tool to study  the dynamics of $ f $ on the boundary of the corresponding Fatou component $ U $ (we consider the boundary restricted to the domain $ \Omega(f) $ of definition of the iterated function $ f $, and  denote it by $ \partial U $). 
Although neither $ g $ nor $ \varphi $ extend continuously to $ \overline{\mathbb{D}} $ in general, some connections can be set in the sense of radial limits and cluster sets  (see Sect. \ref{sect-inner}). Indeed, both measure-theoretical  \cite{Aaronson, DM91} and topological properties \cite{BakerDominguez, Bargmann} of inner functions  have been shown to transfer to the dynamical plane in a wide range of situations, see e.g. 
 \cite{DevaneyGoldberg, BakerDominguez,BakerWeinreich, BaranskiKarpinska_GeometricCodingTree, Bargmann,BFJK_Accesses,FJ23, JF23} and references therein. We note that all the Fatou components studied in the previous works are attracting and parabolic basins, and doubly parabolic Baker domains, which share the same ergodic properties of the boundary map ($ f|_{\partial U} $ is ergodic and recurrent with respect to harmonic measure $ \omega_U $, and $ \omega_U $-almost every point has a dense orbit in $ \partial U $ -- aspects that are used in a fundamental way in the previous works).

On the other hand, hyperbolic and simply parabolic Baker domains may exhibit a completely different boundary behaviour.  For instance, for the simply parabolic Baker domain  of the function	\[f(z)=z+2\pi i \alpha +e^z,\]
for appropriate $ \alpha\in \left[ 0,1\right] \smallsetminus\mathbb{Q} $, studied in  \cite[Thm. 4]{BakerWeinreich} (see also Ex. \ref{ex-sp-univalent}), all points in the boundary, which is a Jordan curve, converge to infinity under iteration. Thus, the previous techniques do not apply and such Baker domains remain somehow unexplored,
except for the results in 
 \cite{RipponStallar_UnivalentBD}, \cite[Thm. A]{BFJK-Escaping}, which establish the measure of the escaping set in $ \partial U $, under certain conditions on the associated inner function.  

The goal of this paper is to understand both topologically and dynamically the boundaries of hyperbolic and simply parabolic Baker domains.

First, from an ergodic point of view, a hyperbolic or simply parabolic inner function is non-ergodic and non-recurrent with respect to the Lebesgue measure $ \lambda $ \cite[Thm. 3.1]{DM91}. By relying on properties of the Riemann map, we show that both non-ergodicity and non-recurrence transfer to the boundary map, with respect to harmonic measure.
\begin{named}{Theorem A}{\bf (Ergodic properties of the boundary map)}\label{teo:A} 
	Let $ f\in\mathbb{K} $, and let $ U $ be a simply connected  Baker domain, of hyperbolic or simply parabolic type. Then, $ f|_{\partial U} $ is non-ergodic and non-recurrent with respect to the harmonic measure $ \omega_U $.
\end{named}

If we restrict ourselves to the cases where the inner function $ g $ is well-defined around its Denjoy-Wolff point $ p\in\partial\mathbb{D} $ (i.e. $ p $ is not a singularity for $ g $), a stronger control on the  dynamics is achieved. Indeed, following the approach taken in \cite{benini2023boundary} to describe the boundary dynamics of wandering domains, one can define the {\em Denjoy-Wolff set} of $ f|_U $ as the set of points $ x\in\partial U $ such that $ f^n(x)\to\infty $ (i.e. $\dist_{\widehat{\mathbb{C}}}(f^n(x), \infty)\to 0$, see \cite[Sect. 9]{benini2023boundary}). In \cite{RipponStallar_UnivalentBD}, \cite[Thm. A]{BFJK-Escaping} it is proven that, if $ U $ is a hyperbolic or simply parabolic Baker domain such that the Denjoy-Wolff point of the associated inner function is not a singularity (in particular, if $ f|_U $ is univalent or has finite degree), then the Denjoy-Wolff set has full harmonic measure.

However, the main limitation of the Denjoy-Wolff set is  that it does not capture in which direction boundary orbits converge to infinity. Indeed, for points inside the Baker domain, the convergence takes place through the same access to infinity (known as the {\em dynamical access}, see \cite{BFJK_Accesses}). Thus, we introduce the notion of {\em Carathéodory set} as the set of points in $ \partial U $ which converge to the image under $ \varphi^* $ of the Denjoy-Wolff point with respect to the Carathéodory topology of $ \partial U $ (or, morally, the points in $ \partial U $ which converge to $ \infty $ through the dynamical access). 
\begin{defi*}{\bf (Carathéodory set)}
		Let $ f\in\mathbb{K}$, and let $ U $ be an invariant simply connected Baker domain. Let $ \varphi\colon\mathbb{D}\to U $ be a Riemann map, and let $ g=\varphi^{-1}\circ f\circ \varphi $ be the inner function associated  with $ f|_U $ by $ \varphi $.  We say that $ x\in\partial U $ is in the {\em Carathéodory set} if, for any crosscut neighbourhood $ N\subset \mathbb{D} $ at the Denjoy-Wolff point $ p\in\partial \mathbb{D} $,  there exists $ k_0\geq0 $ such that, for all $ k\geq k_0 $, \[f^{k}(x)\in\overline{\varphi(N)}.\]
\end{defi*}

By the results of Doering and Mañé for inner functions \cite{DM91}, the Carathéodory set has full harmonic measure (and, in particular, it is dense in $ \partial U $) for simply connected Baker domains, of hyperbolic or simply parabolic type. 
Moreover the escaping points constructed in \cite{RipponStallar_UnivalentBD}, \cite[Thm. A]{BFJK-Escaping} are also in the Carathéodory set (and since they are escaping, they are also in the Denjoy-Wolff set). However,  points in the Carathéodory set may fail to converge to infinity in general, if the cluster set of the Denjoy-Wolff point is non-degenerate. 

In view of these results, one shall ask if there exist points in $ \partial U $ which are not in the Carathéodory set.
The answer is negative in general, as shown by the univalent Baker domain of $ f(z)=z+2\pi i \alpha +e^z $  introduced above (Ex. \ref{ex-sp-univalent}). In Section \ref{subs-caratheodory}, we analyse other examples of univalent Baker domains which have either none or a single non-Carathéodory point in the boundary. We note that, in this case, the associated inner function is a Möbius transformation, and every point in $ \partial \mathbb{D} $ (with at most one exception) converges to the Denjoy-Wolff point. 

On the other hand, if the Baker domain $ U $ is non-univalent, there exists a perfect set in $ \partial\mathbb{D} $ in which iterates do not converge to the Denjoy-Wolff point locally uniformly (the Julia set $ \mathcal{J}(g) $), and thus one expects plenty of non-Carathéodory points in the boundary of $ U $. 

Our result reads as follows, recalling that the {\em post-singular set} of $ f $ is defined as 
\[ P(f)\coloneqq \overline{\bigcup\limits_{s\in SV(f)}\bigcup\limits_{n\geq 0} f^n(s)}, \] where $ SV(f) $ denotes the set of singular values of $ f $.

\begin{named}{Theorem B}{\bf (Non-empty non-Carathéodory set)}\label{teo:B} 
	Let $ f\colon \mathbb{C}\to\mathbb{C}$ be an entire function, and let $ U $ be a non-univalent hyperbolic or simply parabolic Baker domain. Let $ \varphi\colon\mathbb{D}\to U $ be a Riemann map, and let $ g=\varphi^{-1}\circ f\circ \varphi $ be the inner function associated  with $ f|_U $ by $ \varphi $.  
Assume there exists a crosscut neighbourhood $ N_\xi $ of $ \xi\in\mathcal{J}(g) $ such that  $ {\varphi(N_\xi)} \cap P(f) =\emptyset$.  Then, there are uncountably many points which are not in the Carathéodory set, which are furthermore accessible from $ U $.
\end{named}
 
 The hypothesis of the existence of a crosscut neighbourhood $ N_\xi $ of $ \xi\in\mathcal{J}(g) $ such that  $ {\varphi(N_\xi)} \cap P(f) =\emptyset$ is always satisfied when $ SV(f)\cap U $ is compactly contained in $ U $, in particular, if $ f|_U $ has finite degree (see Prop. \ref{prop-compactlySV}). Note also that, if $ SV(f)\cap U $ is compactly contained in $ U $, then the Denjoy-Wolff point of the inner function is not a singularity (Prop. \ref{prop-compactlySV}), and by \cite[Thm. A]{BFJK-Escaping}, the Denjoy-Wolff set (and the Carathéodory set) has full harmonic measure. However, it follows from \ref{teo:B} that not all points in $ \partial U $ are in the Carathéodory set.
 
 The previous result relies strongly on the topology of the boundary of unbounded Fatou components of entire functions, studied in \cite{BakerDominguez, Bargmann, JF23}. 
 Recall that, for non-univalent Fatou components of entire functions, the set of accesses to infinity
\[\Theta\coloneqq \left\lbrace \xi\in\partial \mathbb{D}\colon \varphi^*(\xi)=\infty\right\rbrace \] is related with the Julia set $ \mathcal{J}(g) $ in the sense that $ \mathcal{J}(g) \subset \overline{\Theta}$ (see  Sect. \ref{sect-inner}).
For hyperbolic and simply parabolic Baker domains, this translates into the following properties, concerning the topology of the boundary, the Julia set $ \mathcal{J}(g) $ and the singularities of the associated inner function. 
\begin{named}{Proposition C}{\bf (Baker domains of entire functions)}\label{prop:C} 
	Let $ f\colon \mathbb{C}\to\mathbb{C}$ be an entire function, and let $ U $ be a non-univalent hyperbolic or simply parabolic Baker domain. Let $ \varphi\colon\mathbb{D}\to U $ be a Riemann map, and let $ g=\varphi^{-1}\circ f\circ \varphi $ be the inner function associated  with $ f|_U $ by $ \varphi $.  Then, the following holds.

	\begin{enumerate}[label={\em(\alph*)}]
		\item 	{\em (Topology of $ \partial U $)} $ \partial U $ has infinitely many components. Moreover, for all $ \xi\in\partial \mathbb{D} $, the cluster set $ Cl(\varphi, \xi) \cap\mathbb{C}$ is contained in either one or two  components of $ \partial U $; in the later case, $ \varphi^*(\xi)=\infty $.
		
			\noindent  Each component $ C $ of $ \partial U $ contains points of  the cluster set $ Cl(\varphi, \xi)\cap\mathbb{C} $ of a unique $\xi\in  \mathcal{J}(g) $, with at most countably many exceptions. 
			
		\item  {\em (Associated inner function)} Assume $ \overline{\Theta}\neq \partial\mathbb{D} $. Then, $ \mathcal{J}(g) $ is a Cantor set, and  the set of singularities of $ g $, $ \textrm{\em sing} (g) $, is nowhere dense in $ \partial\mathbb{D} $. If, in addition, the Denjoy-Wolff point $ p\in\partial\mathbb{D} $ is not a singularity, then both $ \mathcal{J}(g) $ and $ \textrm{\em sing} (g)$ have zero $ \lambda $-measure. \label{thmB-item2}
			\item {\em (Periodic points)} Assume $ \mathcal{J}(g)\neq \partial\mathbb{D} $.  Then, periodic points are not dense on $ \partial U $. 
	\end{enumerate}
\end{named}

Examples of non-univalent hyperbolic or simply parabolic Baker domains of entire functions are given in \cite{RipponStallard_BD}, \cite{Rippon_BD}, \cite[Ex. 3.6]{Bargmann} \cite{BergweilerZheng}, compare also with \cite[Sect. 2.5]{Bargmann} for examples of meromorphic functions (thought as inner functions with a unique singularity). The examples in \cite{BergweilerZheng} satisfy that $ \overline{\Theta}\neq \partial\mathbb{D} $, and in fact $ \varphi $ extends continuously to an arc on $ \partial\mathbb{D} $. 

The question on the size of the singularities of the inner function associated to a Fatou component has been widely studied \cite{efjs,FatousAssociates, JF23,Jov24}, see also \cite[Part III]{ivrii2023inner}, although all these inner functions are  associated with attracting or parabolic basins, or doubly parabolic Baker domains.  This is the first time that inner functions of hyperbolic or simply parabolic type are addressed. Note also that, although for finite degree Blaschke products the Julia set is either a Cantor set or the unit circle, this no longer holds  for inner functions of infinite degree \cite[Sect. 2.5]{Bargmann}, and hence {\em \ref{thmB-item2}} provides additional requirements on the inner functions that can be associated to Fatou components.

Finally, one can go one step further and ask whether there are periodic points in the boundary of such Baker domains. Note that it is an open question if, given any Baker domain, there exists a periodic point in its boundary. In a seminal paper
\cite{przytycki_zdunik}, Przytycki and Zdunik proved that accessible periodic points are dense in the  boundaries of basins of rational maps.  This result was extended to unbounded basins and doubly parabolic Baker domains \cite{Jov24}, under certain conditions on the postsingular set.
Even if periodic points are not dense in the boundaries in general, we prove the following.
\begin{named}{Theorem D}{\bf (Boundary dynamics)}\label{teo:D} 
	Let $ f\in\mathbb{K} $ and let $ U $ be a non-univalent simply connected Baker domain, of hyperbolic or simply parabolic  type. Let $ \varphi\colon\mathbb{D}\to U $ be a Riemann map, and let $ g=\varphi^{-1}\circ f\circ \varphi $ be the inner function associated  with $ f|_U $ by $ \varphi $. Assume there exists a crosscut neighbourhood $ N_\xi $ of $ \xi\in\mathcal{J}(g) $ such that  $ \overline{\varphi(N_\xi)} \cap P(f) =\emptyset$.
	
	\noindent Then, the following holds.
	\begin{enumerate}[label={\em(\alph*)}]
		\item There exist countably many accessible periodic points on $ \partial U $. More precisely, for any $ \zeta\in \mathcal{J}(g) $ and any crosscut neighbourhood $ N_\zeta $ of $ \zeta $, there is an accessible periodic point in $ \overline{\varphi(N_\zeta )}$.
	
		\item For any countable collection of crosscut neighbourhoods $ \left\lbrace N_{\xi_k} \right\rbrace _k$ of $ \xi_k\in\mathcal{J}(g) $, there exists an accessible point $ x\in\partial U $ and a sequence $ n_k\to\infty $ such that $ f^{n_k}(x)\in  \overline{\varphi(N_{\xi_k} )}$. Moreover, if $ f|_U $ has infinite degree, $ x $ can be taken {\em bungee}, i.e. such that $ \left\lbrace f^n(x)\right\rbrace _n $ neither escapes nor is compactly contained in $ \Omega (f) $.
			\item If $ \mathcal{J}(g) =\partial\mathbb{D} $ and $ \omega_U(P(f))=0 $, then periodic points are dense on $ \partial U $. Moreover, if $ f|_U $ has infinite degree, there are accessible points on $ \partial U $ whose orbit is dense on $ \partial U $.
	\end{enumerate}
\end{named}

The assumptions in (c) are satisfied for the hyperbolic Baker domain of the function
\[f(z)=2z-3+e^z,\] studied by Bargmann \cite[Ex. 3.6]{Bargmann} (see Ex. \ref{ex-hyp-bargmann}). Note the wide range of boundary dynamics for hyperbolic and simply parabolic Baker domains: from examples for which every point in $ \partial U $ is escaping, to others for which both periodic and bungee points are dense on $ \partial U $.

We remark that the construction of periodic boundary points for basins of rational maps in \cite{przytycki_zdunik} relies strongly on the ergodic properties of the boundary map (more precisely, ergodicity and recurrence), as well as the constructions in \cite{JF23, Jov24}. Instead, we use the topological properties of $ \mathcal{J}(g) $, but the complexity of the proof increases substantially.

\vspace{0.4cm}
{\bf Notation.}
Throughout this article, $ \mathbb{C} $ and $ \widehat{\mathbb{C}} $ denote the complex plane and the Riemann sphere, respectively. Given $ f\in\mathbb{K} $, we denote by $ E(f) $ the set of singularities of $ f $, and $ \Omega(f)\coloneqq\widehat{\mathbb{C}}\smallsetminus E(f) $ (if $ f $ is transcendental entire, $ \Omega(f)=\mathbb{C} $ and $ E(f)=\left\lbrace \infty\right\rbrace  $).
 Given a set $ A\subset \mathbb{C} $, we denote by $ \overline{ A } $ and $ \partial A $, its closure and its boundary taken in $ \Omega(f)$; and by $ \widehat{ A } $ and $ \widehat{\partial} A $, its closure and its boundary when considered in $ \widehat{\mathbb{C}} $.    
 
 If $ U $ is a  simply connected domain, $ \omega_U $ stands for the harmonic measure in $ \partial U $. We denote by $ \mathbb{D} $, the unit disk; by $ \partial\mathbb{D} $, the unit circle; and by $ \lambda $, the Lebesgue measure on $ \partial\mathbb{D} $, normalized so that $ \lambda ( \partial\mathbb{D} )=1 $.
 
 \vspace{0.4cm}
 {\bf Acknowledgements.} I would like to express my deep gratitude to my thesis advisor, Prof. Núria Fagella, for her guidance and support. I would like to thank Phil Rippon and Gwyneth Stallard for helpful discussions and comments, and specially for pointing out the reference \cite{Stallard}, which inspired the proof of \ref{teo:B}.
\section{Inner function associated with a hyperbolic or simply parabolic Baker domain}\label{sect-inner}
The main tool when working with a simply connected invariant Fatou component of $ f\in\mathbb{K} $ is its {\em associated inner function}, defined as \[g\colon \mathbb{D}\to\mathbb{D}, \hspace{0.5cm} g\coloneqq \varphi^{-1}\circ f\circ \varphi,\] for a Riemann map $ \varphi\colon \mathbb{D}\to U $. One can check that $ g $ is indeed an inner function \cite[Prop. 5.6]{Jov24}, meaning that, for $ \lambda $-almost every $ \xi\in\partial\mathbb{D}$, the radial limit \[g^*(\xi)\coloneqq \lim\limits_{r\to 1^-}g(r\xi)\] belongs to $ \partial\mathbb{D} $. Note that $ g $ is unique up to conjugacy by a conformal automorphism of $ \mathbb{D} $. 

In the case when $ f|_U $ has finite degree, then $ g $ is a finite Blaschke product; otherwise $ g $ has infinite degree and has at least a {\em singularity} on $ \partial\mathbb{D} $, i.e. a point  $ \xi\in\partial\mathbb{D} $ such that $ g $ cannot be continued analytically to any neighbourhood of it. We denote the set of singularities of $ g $ by $ \textrm{sing}(g) $. 

Throughout this paper (and specially in Sect. \ref{subsect-inverse-branches}), we assume that every inner function $ g $ is always continued analytically to $ \widehat{\mathbb{C}}\smallsetminus\overline{\mathbb{D}} $ by the reflection principle, and to $ \partial\mathbb{D}\smallsetminus\textrm{sing}(g) $ by analytic continuation. In other words, $ g $ is considered as is maximal meromorphic extension $$ g\colon\widehat{\mathbb{C}}\smallsetminus\textrm{sing}(g) \to\widehat{\mathbb{C}} .$$

\subsection{Iteration of inner functions associated with Baker domains}
Let $ g\colon\mathbb{D}\to\mathbb{D} $ be the inner function associated with a Baker domain. It is clear that $ g|_\mathbb{D} $ has no fixed points; thus, by the Denjoy-Wolff Theorem, there exists a point $ p\in\partial\mathbb{D} $ towards which all points converge under iteration. We say that $ p $ is the {\em Denjoy-Wolff point} of $ g $. It is customary to assume either that the Denjoy-Wolff point is $ 1\in\partial\mathbb{D} $ (by precomposing the Riemann map by a rotation), or to consider the inner function defined in the upper half-plane $ \mathbb{H}$ instead of the unit disk $ \mathbb{D} $, with $ \infty$ being its Denjoy-Wolff point.

The dynamics of holomorphic self-maps of $ \mathbb{D} $ with Denjoy-Wolff point in $ \partial\mathbb{D} $ have been carefully described by Cowen.

\begin{thm}{\bf (Cowen's classification of self-maps of $ \mathbb{D} $, {\normalfont \cite{cowen}})}\label{teo-cowen}
	Let $ g$ be a holomorphic self-map of  $ \mathbb{D} $ with Denjoy-Wolff point $ p\in\partial \mathbb{D} $. Then, there exists a simply connected domain $ V \subset\mathbb{D}$, a domain $ \Omega $ equal to $ \mathbb{C} $ or $ \mathbb{H}=\left\lbrace \textrm{\em Im }z>0\right\rbrace  $, a holomorphic map $ \psi\colon \mathbb{D}\to\Omega  $, and a Möbius transformation $ T\colon\Omega\to\Omega $, such that: \begin{enumerate}[label={\em (\alph*)}]
		\item $ V $ is an {\em absorbing domain} for $ g $ in $ \mathbb{D} $, i.e. $ g(V)\subset V $ and for every compact set $ K\subset\mathbb{D} $, there exists $ n\geq 0 $ such that $ g^n(K)\subset V $;
		\item $ \psi(V) $ is an absorbing domain for $ T $ in $ \Omega $;
		\item $ \psi\circ g=T\circ \psi $ in $ \mathbb{D} $;
		\item $ \psi $ is univalent in $ V $.
	\end{enumerate}
	
	Moreover, $ T $  and $ \Omega $ depend only on the map $ g $, not  on the absorbing domain $ V $. In fact (up to a conjugacy of $ T $ by a Möbius transformation preserving $ \Omega $), one of the following cases holds:\begin{itemize}
		\item $ \Omega= \mathbb{C} $, $ T=\textrm{\em id}_\mathbb{C} +1 $ {\em (doubly parabolic type)},
		\item $ \Omega= \mathbb{H} $, $ T=\lambda\textrm{\em id}_\mathbb{H}$, for some $ \lambda>1 $ {\em (hyperbolic type)},
		\item $ \Omega= \mathbb{H} $, $ T=\textrm{\em id}_\mathbb{H} \pm1 $ {\em (simply parabolic type)}.
	\end{itemize}
\end{thm}

See Figure \ref{figura1}.

In this paper we are interested in hyperbolic and simply parabolic Baker domains, whose associated inner function is, by definition, hyperbolic or simply parabolic, respectively.

\begin{figure}[h]\centering	\captionsetup[subfigure]{labelformat=empty}
\hfill
	\begin{subfigure}[b]{0.24\textwidth}
		\includegraphics[width=\textwidth, ]{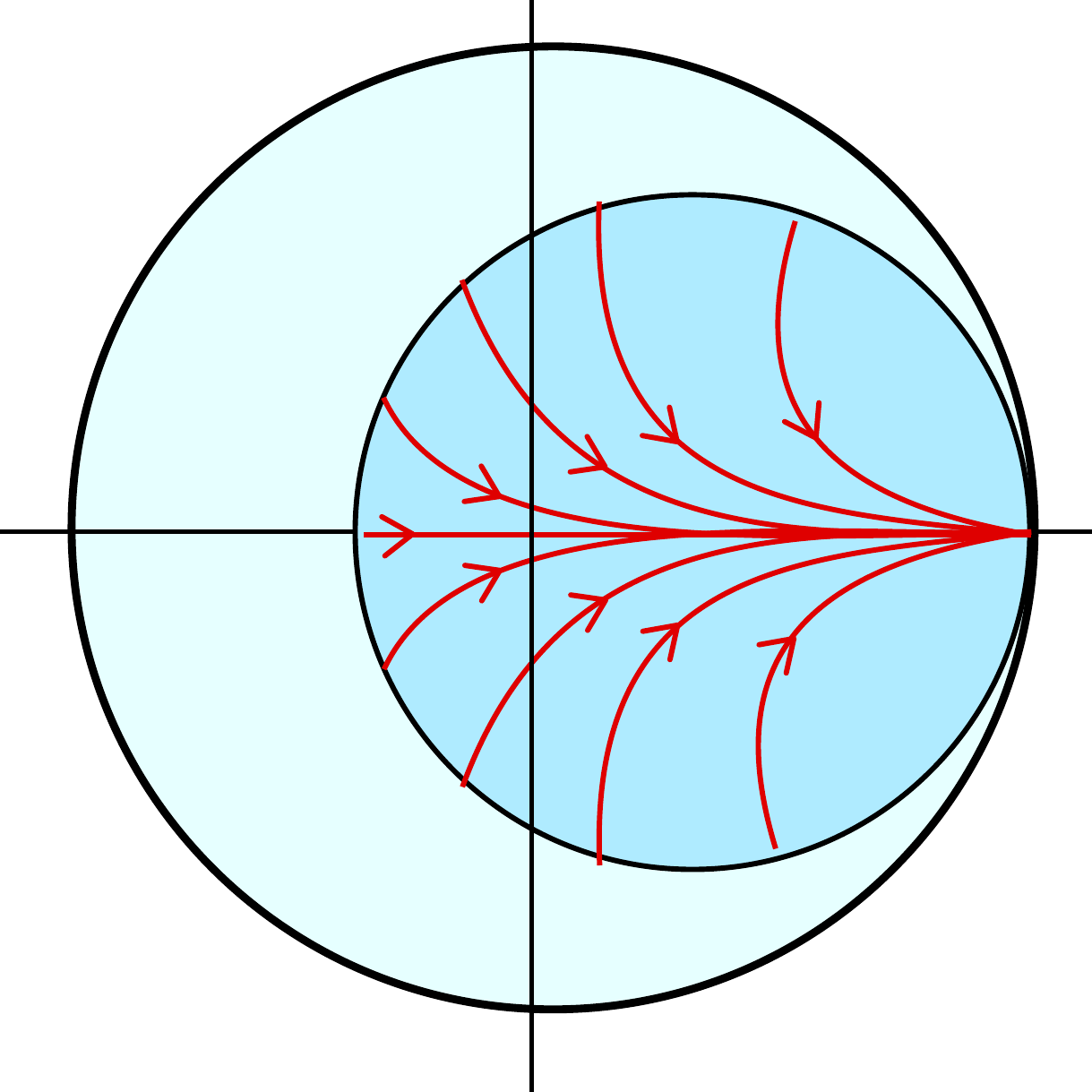}
		\caption{\footnotesize Doubly parabolic}
		
	\end{subfigure}
\hfill
\hfill
	\begin{subfigure}[b]{0.24\textwidth}
		\includegraphics[width=\textwidth]{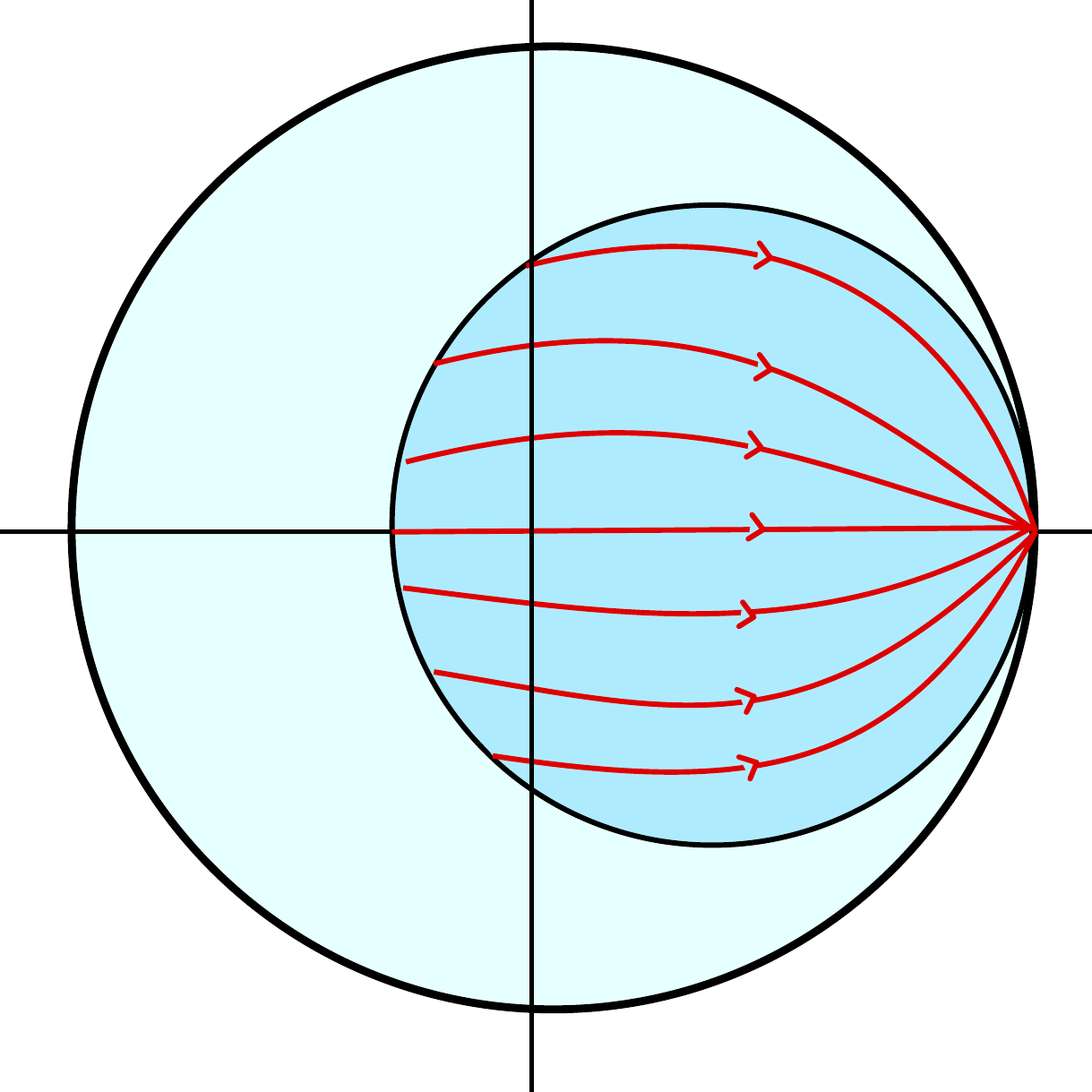}
		\caption{\footnotesize Hyperbolic}
		
	\end{subfigure}
\hfill
\hfill
	\begin{subfigure}[b]{0.24\textwidth}
		\includegraphics[width=\textwidth]{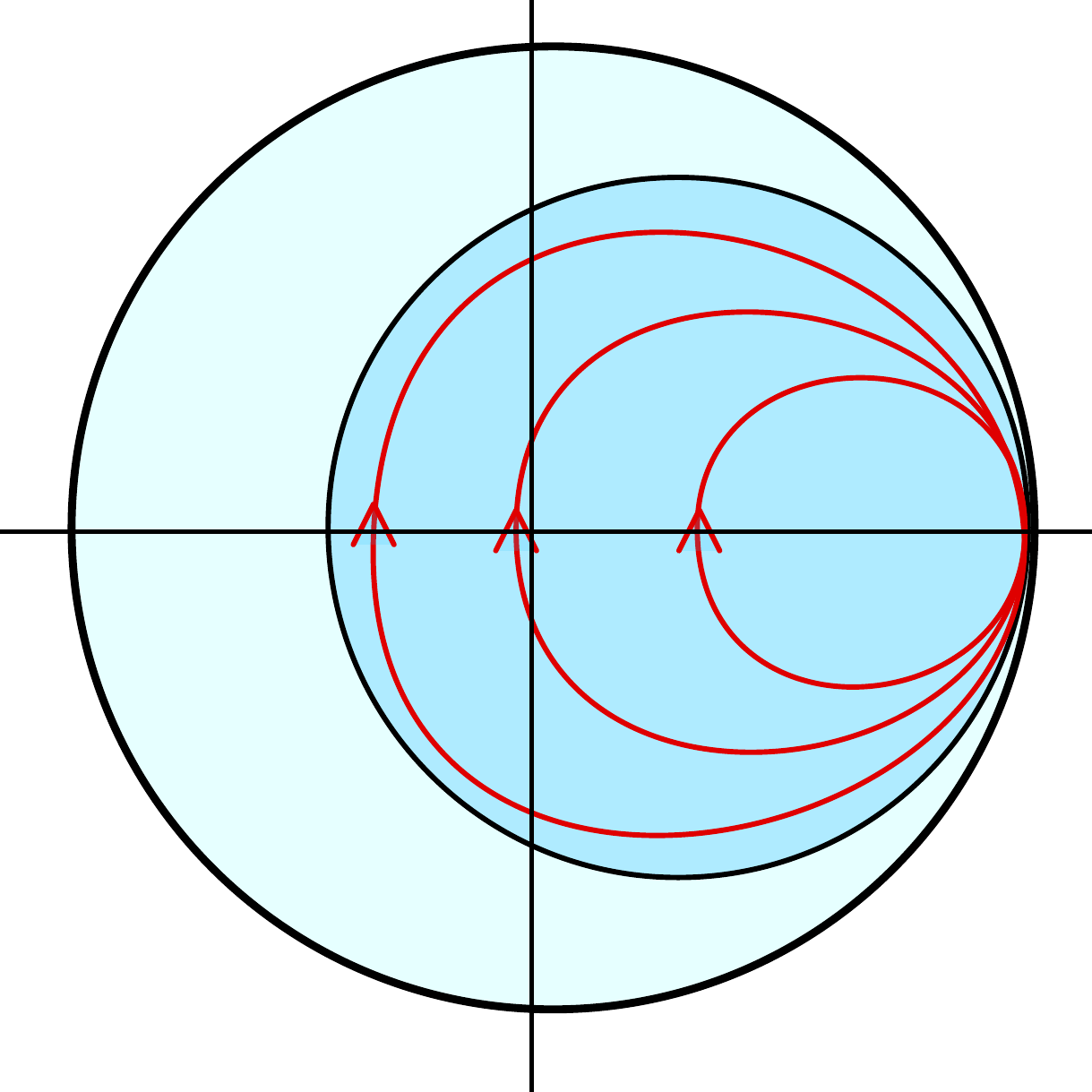}
		\caption{\footnotesize Simply parabolic}
		
	\end{subfigure}
\hfill
	\caption{{\footnotesize The different types of convergence to the Denjoy-Wolff point.}}\label{figura1}
\end{figure}

\subsubsection*{Ergodic properties of $ g^*\colon\partial\mathbb{D}\to\partial\mathbb{D} $} Any inner function $ g\colon\mathbb{D}\to\mathbb{D} $ induces a measure-theoretical dynamical system on $ \partial\mathbb{D} $, given by considering its radial extension\[g^*\colon\partial \mathbb{D}\to\partial\mathbb{D},\] which is well-defined $ \lambda $-almost everywhere and measurable. In \cite{Aaronson, DM91}, ergodic properties of such maps have been studied (see also \cite[Sect. 3.2]{Jov24}). 

Recall that given a measure space $ (X,\mathcal{A}, \mu) $ and a measurable transformation $ T\colon X\to X $ (not necessarily measure-preserving), $ T $ is said to be {\em ergodic} if $ T $ is non-singular and for every $ A\in\mathcal{A} $ with $ T^{-1}(A)=A $, it holds $ \mu (A)=0 $ or $ \mu (X\smallsetminus A)=0 $. The map $ T $ is {\em recurrent} if for every $ A\in\mathcal{A} $ and $ \mu $-almost every $ x\in A $ there exists a sequence $ n_k\to\infty $ such that $ T^{n_k}(x)\in A $. 

The radial extension of an inner function is always non-singular \cite[Prop. 6.1.1]{Aaronson}. Moreover, for hyperbolic and simply parabolic inner functions, the following holds (see \cite[Thm. 3.1, Thm. 4.1, Corol. 4.3]{DM91}, combined with \cite[Lemma 2.6]{Bargmann}).
\begin{thm}{\bf (Ergodic properties of $ g^* $)}\label{thm-ergodic-g}
	Let $ g\colon \mathbb{D}\to\mathbb{D} $ be an inner function, of hyperbolic or simply parabolic type. Then, $ g^*|_{\partial\mathbb{D}} $ is non-ergodic and non-recurrent. Moreover, $ \lambda $-almost every $ \xi \in\partial\mathbb{D}$ converges to the Denjoy-Wolff point under the iteration of $ g^* $.
\end{thm}

Moreover, we need the following result, which is extracted from the proof of \cite[Thm. 3.1]{DM91}.
\begin{thm}\label{thm-DoreingManeInner}
	Let $ g\colon \mathbb{D}\to\mathbb{D} $ be a hyperbolic or simply parabolic inner function. Then, there exists $ h\colon\mathbb{D}\to\mathbb{D} $ and a Möbius transformation $ T\colon\mathbb{D}\to\mathbb{D} $ such that $ h\circ g= T\circ h$. Moreover, $ h=\lim_n h_n $, where $ h_n=T_n\circ g^n $, $ T_n\colon\mathbb{D}\to\mathbb{D} $ Möbius,  and $ h_n(0)=0 $ for all $ n $.
\end{thm}
 It follows from \cite{Ferreira23} that $ h $ is inner (and $ h(0)=0 $, and hence preserves the Lebesgue measure on $ \partial \mathbb{D} $, \cite[Thm. A]{DM91}); moreover, $ h $ is a Möbius transformation if and only if $ g $ is univalent, and otherwise $ h $ has infinite degree. Note that $ T $ cannot have fixed points, and thus is hyperbolic or simply parabolic (doubly parabolic inner functions are never univalent).

Applying the Lehto-Virtanen Theorem \cite[Sect. 4.1]{Pom92} as usual, one deduces that $ h^* $ and $ h^*\circ g^* $ exists $ \lambda $-almost everywhere, and \[h^*\circ g^*=T\circ h^*,\]where defined.
\subsubsection*{The Fatou set $ \mathcal{F}(g) $ and the Julia set $ \mathcal{J}(g) $} Another approach to describe the dynamics of inner functions on $ \partial\mathbb{D} $ is developed in \cite{BakerDominguez, Bargmann}, where instead of considering the induced measure-theoretical dynamical system as in \cite{Aaronson, DM91}, it is considered the holomorphic dynamical system given by the maximal meromorphic extension of $ g $, and the normality of the sequences of iterates is studied.

 More precisely, assume $ g $ is not a Möbius transformation. Then, the {\em Fatou set} $ \mathcal{F}(g) $ is the set of all points $ z\in\widehat{\mathbb{C}} $ for which there exists an open neighbourhood $ V\subset \widehat{\mathbb{C}} $ of $ z $ such that $ \left\lbrace g^n|_V\right\rbrace _n $ is well-defined and normal. The {\em Julia set} $ \mathcal{J}(g) $ is the complement of $ \mathcal{F} (g)$ in $ \widehat{\mathbb{C}} $. In view of the Denjoy-Wolff Theorem, it is clear that the Fatou set is precisely the set of points for which iterates converge locally uniformly to the Denjoy-Wolff point.
 
 It follows from Montel's theorem that $ \mathcal{J}(g)\subset \partial\mathbb{D} $. Note also that, if $ g $ has finite degree, then it is a Blaschke product, and thus a rational map. In this case, the definition of the Fatou and Julia sets agrees with the usual one for rational functions. Moreover, the following holds.
 \begin{lemma}{\bf (Properties of Fatou and Julia sets, {\normalfont\cite[Lemma 8]{BakerDominguez}, \cite[Thm. 2.34]{Bargmann}
 		})}\label{lemma-properties-inner-function}
 	Let $ g\colon\mathbb{D}\to\mathbb{D} $ be an inner function. Then,
 	\begin{enumerate}[label={\em(\alph*)}]
 		\item $ g(\mathcal{F}(g))\subset \mathcal{F}(g) $;
 		\item if $ g $ is non-Möbius, then $ \mathcal{J}(g) $ is a perfect set;
 		\item if $ g $ is non-rational, then $ \mathcal{J}(g)=\overline{\bigcup_{n} \textrm{\em sing}(g^n)}$.
 	\end{enumerate}
 \end{lemma}

It is well-known that, if $ g $ is doubly parabolic, then $ \mathcal{J}(g)=\partial\mathbb{D} $ \cite[Thm. 2.24]{Bargmann}. For non-Möbius hyperbolic and simply parabolic inner functions, it may happen $ \mathcal{J}(g)=\partial\mathbb{D} $, as well as $ \mathcal{J}(g)\neq\partial\mathbb{D} $, see \cite[Sect. 2.5]{Bargmann}  (both situations can happen also for inner functions associated with non-univalent Baker domains of entire function; see \cite[Ex. 3.6]{Bargmann}, where $  \mathcal{J}(g)=\partial\mathbb{D} $; and \cite{BergweilerZheng}, where $  \mathcal{J}(g)\neq\partial\mathbb{D} $).

We need the following additional properties.
\begin{lemma}{\bf (Preimages near singularities,  {\normalfont \cite[Lemma 5 and Corollary]{BakerDominguez}})}\label{lemma-sing}
	Let $ g\colon\mathbb{D}\to\mathbb{D} $ be an inner function, and let $ \zeta\in\textrm{\em sing}(g) $.  Then, for every $ \xi\in\partial\mathbb{D} $ and every neighbourhood $ U $ of $ \zeta $, there exists $ \eta\in U\cap\partial\mathbb{D} $ such that $ g^*(\eta)=\xi $.
\end{lemma}

\begin{lemma}{\bf (Iterated preimages are dense in the Julia set)}\label{lemma-preimages-inner}
Let $ g\colon\mathbb{D}\to\mathbb{D} $ be a non-Möbius inner function, and let $ \zeta\in\partial\mathbb{D}$. Then, 
	\[\overline{\bigcup_{n\geq 0} \left\lbrace \xi\in\partial\mathbb{D}\colon (g^n)^*(\xi)=\zeta\right\rbrace} \supset\mathcal{J}(g).\]
\end{lemma}
\begin{proof}
	In the case when $ g $ has finite degree, it follows from the standard theory of Fatou and Julia sets. In the case when $ g $ has infinite degree, there exists at least one singularity of $ g $, and \[\mathcal{J}(g)=\overline{\bigcup\limits_{n\geq0} \textrm{ sing}(g^n)},\] by Lemma \ref{lemma-properties-inner-function}(c). By Lemma \ref{lemma-sing}, each singularity is approximated by radial preimages of every point in $ \partial\mathbb{D} $, and the lemma follows.
\end{proof}

\subsection{Inverse branches for inner functions at points in the unit circle}\label{subsect-inverse-branches}

Far from the singularities, the inner function $ g $ is holomorphic and exhibits a strong symmetry, since it maps the unit circle to itself locally conformally. This can be used to give precise estimates on the distortion of the radial segment in terms of Stolz angles under  the inverse branches of an inner function.

In this section, consider the inner function to be defined in the upper half-plane, that is, $ h\colon\mathbb{H}\to\mathbb{H} $, and let $ \infty $ be the Denjoy-Wolff point of $ h $. We use the same notation $ h $ for the maximal meromorphic extension of the inner function\[h\colon\widehat{\mathbb{C}}\smallsetminus\textrm{sing}(h)\to \widehat{\mathbb{C}}.\]

As usual, we say that $ v\in\widehat{\mathbb{C}} $ is a {\em regular value} of $ h $  if there exists $ \rho>0 $ such that all inverse branches $ H_1 $ of $ h $ are well-defined in $ D(v,\rho) $.
{\em Singular values} of $ h $, denoted by $ SV(h) $, are defined as the set of values which are not regular. It can be seen that the set of singular values coincides with the closure of the set of critical and asymptotic values  (see \cite[Sect. 2.6, 4]{Jov24} for a wider explanation on singular values of inner functions). Since $ h $ is symmetric along the real line, one shall consider
\[SV(h,\mathbb{H})\coloneqq SV(h)\cap\mathbb{H}.\]
Note that, for inner functions associated to Fatou components, the following inclusion holds
\[ \varphi(SV(h)\cap\mathbb{H})\subset SV(f)\cap U.\]

\begin{prop}{\bf (Singular values in $ \mathbb{R} $,	{\normalfont\cite[Prop. 4.2]{Jov24}})}\label{prop-SVinner}
Let $ h\colon\mathbb{H}\to\mathbb{H} $ be an inner function, and let $ x\in\mathbb{R} $. The following are equivalent.
\begin{enumerate}[label={\em(\alph*)}]
	\item There exists a crosscut $ C $, with crosscut neighbourhood $ N_C $, such that $ x\in\partial N_C $ and $ SV(h, \mathbb{H})\cap N_C=\emptyset $.
	\item $ x $ is regular.
\end{enumerate}
\end{prop}
Given the singular values $ SV(h, \mathbb{H}) $, one can consider the {\em postsingular set} \[ P(h, \mathbb{H})\coloneqq\overline{\bigcup\limits_{s\in SV(h, \mathbb{H})}\bigcup\limits_{n\geq 0} h^n(s)}.\] It is clear that if there exists a crosscut neighbourhood of $ x\in \mathbb{R} $ disjoint from $ P(h, \mathbb{H}) $, then all inverse branches $ H_n $ of $ h^n $ are well-defined in some disk around $ x $.
Moreover, the following holds.
\begin{prop}{\bf (Compactly contained singular values,	{\normalfont\cite[Corol. 4.3, Prop. 5.2]{Jov24}})}\label{prop-compactlySV}
	Let $ h\colon\mathbb{H}\to\mathbb{H} $ be an inner function, with Denjoy-Wolff point $ \infty\in\partial\mathbb{H} $. Assume singular values of $ h $ in $ \mathbb{H} $ are compactly contained in $ \mathbb{H} $. Then, there exists $ r>0 $ such that for every $ x\in  \mathbb{R}$,  $D(x,r)\cap P(h)=\emptyset $. Moreover, the Denjoy-Wolff point is not a singularity for $ h $.
\end{prop}

Once the domain of conformality of inverse branches is established, let us study their distortion. To that end, let us consider the {\em radial segment} and the {\em Stolz angle} at a point $ x\in \mathbb{R} $, that is
\[R_\rho(x)\coloneqq \left\lbrace z\in\mathbb{H}\colon\ \textrm{Im }w<\rho,\ \textrm{Re }w =x\right\rbrace ,\]
\[\Delta_{\rho, \alpha}(x)\coloneqq \left\lbrace z\in\mathbb{H}\colon\ \textrm{Im }w<\rho, \ \frac{\left| \textrm{Re }w-x \right| }{\textrm{Im }w}<\alpha\right\rbrace .\]
The following holds.

\begin{prop}{\bf (Control of radial limits in terms of Stolz angles, {\normalfont\cite[Sect. 4]{Jov24}})}\label{prop-radial-limits}
	Let $ h\colon\mathbb{H}\to\mathbb{H} $ be an inner function, with Denjoy-Wolff point $ \infty\in\partial{\mathbb{H}} $. Let $ x\in\mathbb{R} $.
	Assume there exists $ \rho_0>0$ such that $ D(x, \rho_0)\cap P(h)\neq \emptyset$.  Then, for all $ \alpha\in (0, \frac{\pi}{2}) $, there exists $ \rho_1\coloneqq \rho_1(\alpha, \rho_0)<\rho_0 $ such that all branches $ H_n $ of $ h^{-n} $ are well-defined in $ D(x, \rho_1) $ and, for all $ \rho<\rho_1 $,\[H_n(R_\rho(x))\subset \Delta_{\alpha,\rho}(H_n(x)).\]
\end{prop}

Throughout the paper we shall use Proposition \ref{prop-radial-limits} for inner functions defined in the unit disk. Defining the radial segment and the Stolz angle as its preimages by the conformal map\[M\colon\mathbb{D}\to\mathbb{H}\hspace{0.5cm} z\mapsto i\frac{1+z}{1-z},\] the same inclusion relations are satisfied. We use the same symbols to denote the radial segment and the Stolz angles in $ \mathbb{D} $ and in $ \mathbb{H} $; this should not cause any trouble to the reader since it is clear from the context. See Figure \ref{fig-SVinner}.

 	\begin{figure}[htb!]\centering
	\includegraphics[width=16cm]{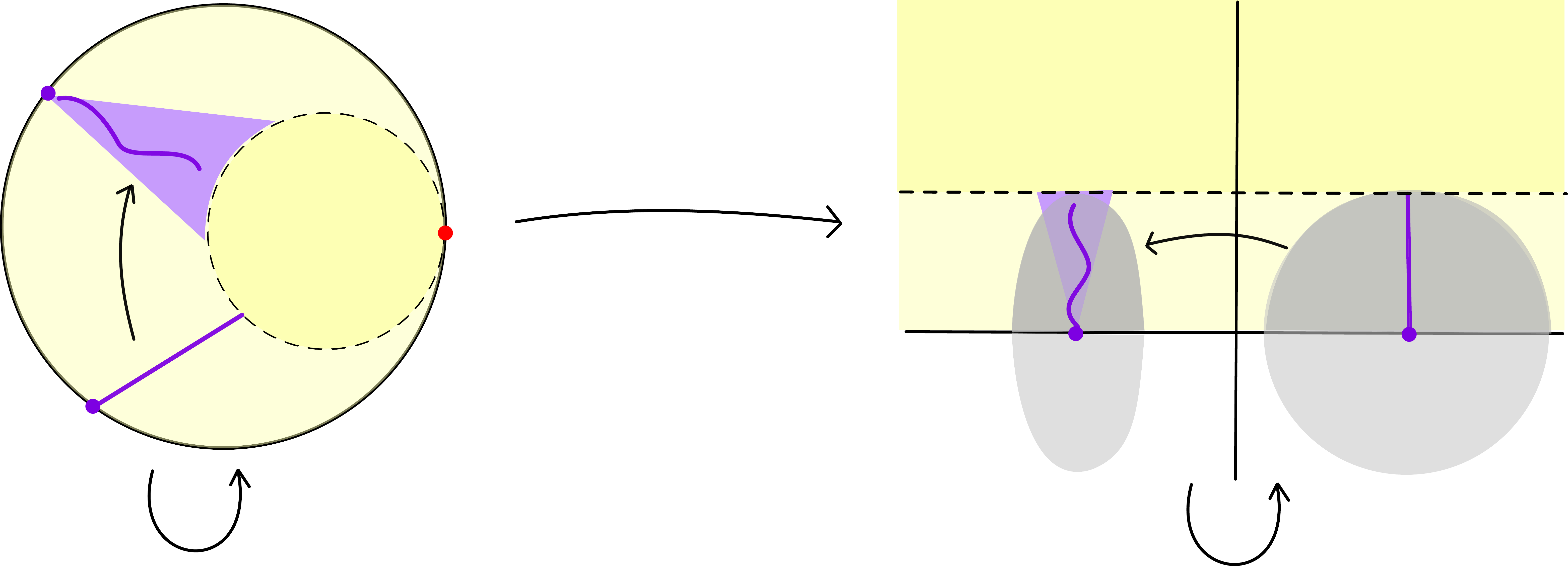}
	\setlength{\unitlength}{16cm}
	\put(-0.71, 0.21){\small$1$}
	\put(-0.88, -0.01){$g$}
	\put(-0.22, -0.02){$h$}
	\put(-0.6, 0.24){$M$}
		\put(-0.62, 0.2){\footnotesize$z\mapsto i\frac{1+z}{1-z}$}
	\put(-0.958, 0.33){$\mathbb{D}$}
	\put(-0.03, 0.34){$ \mathbb{H}$}
	\put(-0.24, 0.215){\small$H_1$}
		\put(-0.11, 0.125){\small$x$}
			\put(-0.33, 0.12){\small$H_1(x)$}
	\put(-0.95, 0.19){\small$G_1$}
	\caption{\footnotesize Consider the inner function $ h\colon\mathbb{H}\to\mathbb{H} $, with Denjoy-Wolff point $ \infty $. Assuming that there are no singular values in some crosscut neighbourhood of $  x\in\mathbb{R}$ (grey), inverse branches are well-defined in a disk around $ x $ (light grey), and one can control the distortion on the radial segment in terms of Stolz angles (purple). The results transfer straight-forward to the unit circle by means of the Möbius transformation $ M $.}\label{fig-SVinner}
\end{figure}
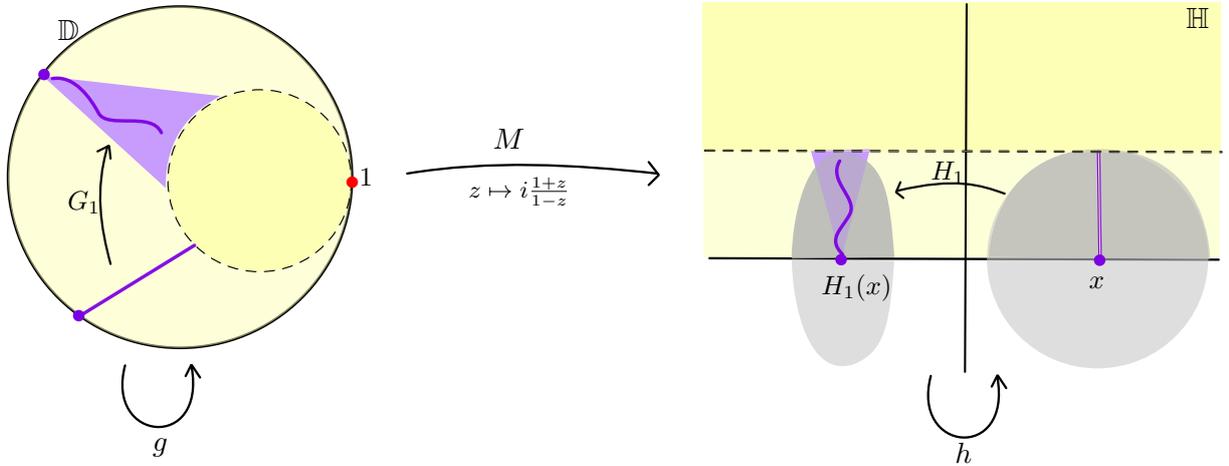

\subsection{Boundary behaviour of the Riemann map}\label{subsect-Riemann} In order to transfer the previous results on the iteration of inner functions on the unit circle $ \partial \mathbb{D} $, we need a deep understanding of  the Riemann map $ \varphi\colon\mathbb{D}\to U $. 
We collect here the results concerning the boundary behaviour of Riemann maps needed in this paper (concerning both general domains and simply connected Fatou components), and refer to \cite{Pom92} for a wider exposition on the topic. In the sequel we assume $ U\subsetneq\mathbb{C} $; this can be achieved without loss of generality by postcomposing $ \varphi $ by a Möbius transformation (or, in the case of Fatou components of functions in class $ \mathbb{K} $, by conjugating  by a Möbius transformation).

\begin{defi}{\bf (Radial limits and cluster sets)}\label{defi-radial-lim}
	Let $ \varphi\colon \mathbb{D}\to U \subsetneq\mathbb{C}$ be a Riemann map and let $ \xi\in\partial \mathbb{D} $. 
	\begin{itemize}
		\item The {\em radial limit} of $ \varphi $ at $ \xi $ is defined to be $\varphi^* (\xi )\coloneqq\lim\limits_{r\to 1^-}\varphi(r\xi)$.
		\item $ \varphi $ has {\em angular limit} $ \varphi^* (\xi )$ if, for any Stolz angle $ \Delta $ at $ \xi $, $ f(z)\to \varphi^* (\xi ) $  when $ z\to\xi $ in $ \Delta $.
		\item The {\em cluster set} $ Cl(\varphi, \xi) $ of $ \varphi $ at $ \xi $ is the set of values $ w\in\widehat{\mathbb{C}} $ for which there is a sequence $ \left\lbrace z_n\right\rbrace _n \subset\mathbb{D}$ such that $ z_n\to \xi $ and $\varphi (z_n)\to w $, as $ n\to\infty $.
		
		\noindent More generally, for $ K\subset\partial \mathbb{D} $, the {\em cluster set} $ Cl(\varphi, K) $ is the set of values $ w\in\widehat{\mathbb{C}} $ for which there is a sequence $ \left\lbrace z_n\right\rbrace _n \subset\mathbb{D}$ such that $ z_n\to K$ and $\varphi (z_n)\to w $, as $ n\to\infty $.
	\end{itemize}
\end{defi}

It follows from the Lehto-Virtanen Theorem \cite[Sect. 4.1]{Pom92} that the angular limit exists whenever the radial limit exists. The following is  a classical result by Beurling \cite[Thm. 9.19]{Pom92}.
\begin{thm}{\bf (Existence of radial limits)}\label{thm-capacity-Riemann-map}
	Let $ \varphi\colon\mathbb{D}\to U\subsetneq\mathbb{C} $ be a Riemann map. Then, for all $ \xi\in\partial\mathbb{D} $ apart from a set of logarithmic capacity zero, the radial limit $ \varphi^* $ exists and its finite.
\end{thm}

In particular, radial limits exist and are different $ \lambda $-almost everywhere.

We shall not discuss here the concept of logarithmic capacity (for which we refer to \cite[Chap. 9]{Pom92}), but keep in mind the following properties. First,  the notion of logarithmic capacity is usually defined  for compact subsets of the plane, but the notion extends to Borel sets \cite[Thm. 9.12]{Pom92}; in particular, a Borel set has logarithmic capacity zero if it does not contain any compact set of positive capacity.
Recall that sets of logarithmic capacity zero are extremely thin: they cannot contain non-degenerate continua, and its Hausdorff dimension is zero \cite[Thm. 10.1.3]{Pom92}. Moreover, the union of countably many sets of capacity zero has capacity zero \cite[Corol. 9.13]{Pom92}, and, if $ \varphi $ is a Möbius transformation and $ E $ has capacity zero, then $ \varphi(E) $ has also capacity zero.

\begin{defi}{\bf (Accessible point)}
	Given an open subset $ U\subset\widehat{\mathbb{C}} $, 
	a point $ v\in\widehat{\partial} U $ is \textit{accessible} from $ U $ if there is a path $ \gamma\colon \left[ 0,1\right) \to U $ such that $ \lim\limits_{t\to 1^-} \gamma(t)=v $. We also say that $ \gamma $ \textit{lands} at $ v $. 
\end{defi}

It is clear that $ v\in\widehat{\partial} U $ with $ \varphi^*(\xi) =v$ for some $ \xi\in\partial \mathbb{D} $ is accessible. The following result bounds the set of  points on the unit circle which have the same radial limit.

\begin{thm}{\bf({\normalfont \cite[Corol. 2.19]{Pom92}})}\label{thm-pommerenke-biaccess} Let $ \varphi\colon \mathbb{D} \to\mathbb{C} $ be a homeomorphism. Then, there are at most countably many points $ a\in\widehat{\mathbb{C}} $ such that $ \varphi^*(\xi_j)=a $ for three distinct points $ \xi_j\in\partial\mathbb{D} $.
\end{thm} 



Finally, given a Riemann map $ \varphi\colon\mathbb{D}\to U $, $ \varphi(0)=z\in U $, the {\em harmonic measure} $ \omega_U(z, \cdot) $, defined on the Borel $ \sigma $-algebra of $ \partial U $ is the push-forward of the normalized Lebesgue measure on $ \partial \mathbb{D} $,  denoted by $ \lambda $. Note that it depends on the basepoint $ z $; however,  since we are only interested in studying ergodicity and recurrence of $ f|_{\partial U} $ (which only depend on sets of zero mesure, which in their turn do not depend on the basepoint),  we consider the harmonic measure on $ \partial U $ without specifying the basepoint, and denote it by $ \omega_U $.

\subsubsection*{Boundary behaviour of the Riemann map for Fatou components}
In the particular case of Fatou components,  the radial extension of the  Riemann map allows to extend the conjugacy $ \varphi \circ f=g\circ \varphi$ to the boundary. First note the following direct consequence of Theorem \ref{thm-capacity-Riemann-map} (note that singularities of $ f\in\mathbb{K} $ are countable by definition).
\begin{lemma}{\bf (Existence of radial limits)}\label{lemma-existence-radial-limits}
	Let $ f\in\mathbb{K} $, let $ U $ be a simply connected Fatou component, and let $ \varphi\colon \mathbb{D}\to U $ be a Riemann map. Then, for all $ \xi\in\partial\mathbb{D} $ apart from a set of logarithmic capacity zero, the radial limit $ \varphi^* $ exists and belongs to $ \Omega (f) $.
\end{lemma}

Then, the conjugacy $ \varphi \circ f=g\circ \varphi$  extends to the boundary as follows: if $ \xi\in\partial\mathbb{D} $ and $ \varphi^*(\xi) \in\Omega(f)$, then $ g^*(\xi)$ and $ \varphi^*(g^*(\xi)) $ exist, and \[ f(\varphi^*(\xi))=\varphi^*(g^*(\xi))\] \cite[Lemma 5.5]{Jov24}.
In other words, the diagram \[\begin{tikzcd}
\partial	U \arrow{r}{f} & \partial U \\	
	\mathbb{D} \arrow{r}{g^*} \arrow{u}{\varphi^*} & \mathbb{D}\arrow[swap]{u}{\varphi^*}
\end{tikzcd}
\] commutes $ \lambda $-almost everywhere (in fact, outside a set of logarithmic capacity zero). 

Although the Julia set $ \mathcal{J}(g) $ can be considered for any inner function, it is specially relevant in the case where $ g $ is the inner function associated with an unbounded Fatou component of an entire function, since it is strongly related with the set of accesses to infinity, 
\[\Theta\coloneqq \left\lbrace \xi\in\partial \mathbb{D}\colon \varphi^*(\xi)=\infty\right\rbrace. \] 
Indeed, the following holds.
\begin{thm}{\bf (Accesses to infinity and Julia sets, {\normalfont \cite[Lemma 13]{BakerDominguez}})}\label{thm-accesses-and-julia}
	Let $ f\colon\mathbb{C}\to\mathbb{C} $ be an entire function, and let $ U $ be a non-univalent Baker domain. Then, $ \mathcal{J}(g)\subset \overline{\Theta} $.
\end{thm}
\begin{thm}{\bf (Accesses to boundary points, {\normalfont \cite[Thm. 3.8]{Bargmann}})}\label{thm-bargmann-unique-accessibility}
	Let $ f\colon\mathbb{C}\to\mathbb{C} $ be an entire function, and let $ U $ be a  Baker domain. Let $ \varphi\colon \mathbb{D}\to U $ be a Riemann map.  Let $ \Theta_\mathbb{C}$ be the set of all $ \xi\in\partial\mathbb{D} $ such that $ \varphi^*(\xi) $ exists and it is finite; and denote by $ AP(U) $  the set of finite accessible points on $ \partial U\subset\mathbb{C} $.
	
	\noindent Then, the map \[ \Theta_\mathbb{C}\to AP(U), \hspace{0.5cm}\xi\mapsto\varphi^*(\xi),\] is a bijection.
\end{thm}
\begin{prop}{\bf (Disconnected cluster sets, {\normalfont\cite[Prop. 3.10]{JF23}})}\label{lemma-disconnected-cluster-sets}		Let $ f\colon\mathbb{C}\to\mathbb{C} $ be an entire function, and let $ U $ be a  Baker domain. Let $ \varphi\colon\mathbb{D}\to U $ be a Riemann map. Let $\xi\in\partial \mathbb{D} $ be such that $ Cl(\varphi, \xi)\cap\mathbb{C} $ is contained in more than one component of $ \partial U $. Then, $ \varphi^*(\xi)=\infty $, and $ Cl(\varphi, \xi)\cap\mathbb{C} $ is contained in exactly two components of $ \partial U $.
\end{prop}
\section{Ergodic properties of $ f\colon \partial U\to\partial U $. Proof of \ref{teo:A}}\label{sect-ergodic-properties}
Now we prove \ref{teo:A}, which states that, for a hyperbolic or simply parabolic Baker domain $ U $, the boundary map $ f\colon\partial U\to \partial U$ is non-ergodic and non-recurrent with respect to harmonic measure $ \omega_U $.

Note that, even though 
the diagram \[\begin{tikzcd}
	\partial	U \arrow{r}{f} & \partial U \\	
\partial	\mathbb{D} \arrow{r}{g^*} \arrow{u}{\varphi^*} & \partial \mathbb{D}\arrow[swap]{u}{\varphi^*}
\end{tikzcd}
\] commutes $ \lambda $-almost everywhere,  and thus $ f|_{\partial U} $ is a factor of $ g^*|_{\partial \mathbb{D}} $, the map $ \varphi^*\colon\partial\mathbb{D}\to\partial U $ need not be an isomorphism. Thus, the non-ergodicity and non-recurrence of $ f|_{\partial U} $ can not be deduced straightforward from  the non-ergodicity and non-recurrence of $ g^*|_{\partial \mathbb{D}} $ (Thm. \ref{thm-ergodic-g}), and our proof relies on properties of the Riemann map (more precisely, Thm. \ref{thm-pommerenke-biaccess}), and of hyperbolic and simply parabolic inner functions (Thm. \ref{thm-ergodic-g}, \ref{thm-DoreingManeInner}).  For more details on the ergodic properties of measure-theoretical isomorphisms and factor maps see e.g. \cite[Sect. 4.1.g]{KatokHasselblat}.

\begin{obs}
	In the particular case of hyperbolic and simply parabolic Baker domains of  entire maps, the map $ \varphi^*\colon\partial\mathbb{D}\to\partial U $ is one-to-one (Thm. \ref{thm-bargmann-unique-accessibility}), and thus the non-ergodicity and non-recurrence follow straightforward from the analogous properties of the associated inner function.
\end{obs}
\begin{proof}[Proof of \ref{teo:A}]
	Let us start by proving that $ f|_{\partial U} $ is non-recurrent with respect to $ \omega_U $. Note that if \[\varphi^*\colon\partial\mathbb{D}\to\partial U\] is a measure-theoretical isomorphism (i.e. a bijection up to sets of zero measure), non-recurrence follows from the same property of the associated inner function. Therefore, we shall assume that there exist $ \xi_1, \xi_2\in\partial\mathbb{D} $ and $ x\in\partial U\subset\Omega(f) $ such that $ \varphi^*(\xi_1) =\varphi^*(\xi_2)=x $. Note that neither $ \xi_1 $ nor $ \xi_2 $ is the Denjoy-Wolff point of $ g $ (otherwise their radial limit would be $ \infty $).
	
	 Then, the image under $ \varphi $ of the radial segments at $ \xi_1 $ and $ \xi_2 $, together with $ x $, i.e.
	\[\varphi(R(\xi_1))\cup\varphi(R(\xi_2))\cup \left\lbrace x\right\rbrace \] is a Jordan curve in $ \widehat{\mathbb{C}} $, and divides $ \partial U $ in two sets $ X_1, X_2 $ of positive measure. Moreover, $ \xi_1 $ and $ \xi_2 $ delimit two (non-degenerate) open circular arcs, say  $ I_1 $ and $ I_2 $, and  $ (\varphi^*)^{-1}(X_1) \subset I_1$ and $ (\varphi^*)^{-1}(X_2) \subset I_2$. We assume  the Denjoy-Wolff point of the inner function lies in  $ I_1 $.  Since $ \lambda $-almost $ \xi\in\partial\mathbb{D} $ converges to the Denjoy-Wolff point under the iteration of $ g^* $ (Thm. \ref{thm-ergodic-g}), 
	it follows that  $ \omega_U $-almost every point in  $ X_2 $ does not come back to $ X_2 $ infinitely often (since its orbit is eventually contained in $ X_1 $), and proves non-recurrence. See Figure \ref{fig-recurrencia}.
	
	 	\begin{figure}[htb!]\centering
		\includegraphics[width=15cm]{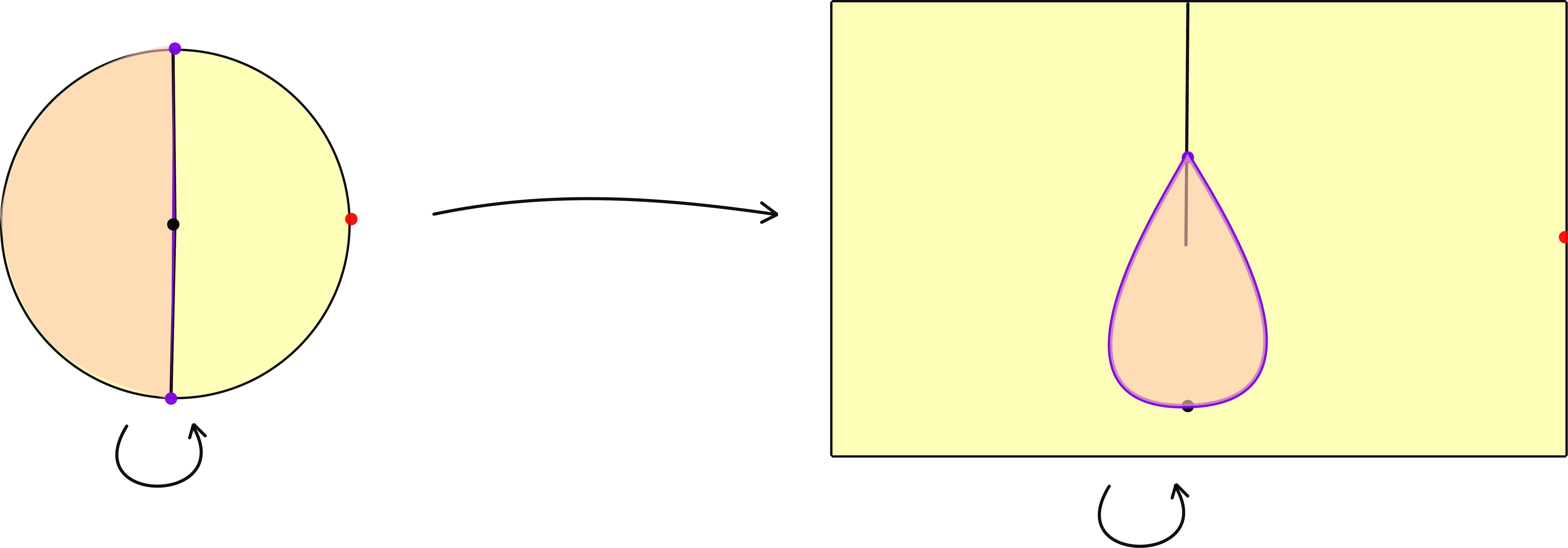}
		\setlength{\unitlength}{15cm}
		\put(-0.765, 0.2){\small$1$}
		\put(-0.91, 0.02){$g$}
		\put(-0.28, -0.02){$f$}
		\put(-0.6, 0.24){$\varphi$}
		\put(-0.97, 0.3){$\mathbb{D}$}
		\put(-0.03, 0.32){$ U$}
		\put(-0.235, 0.25){\small$x\in\partial U$}
		\put(-0.065, 0.195){\small$\varphi^*(1)$}
		\put(-0.265, 0.07){\small$\varphi(0)$}
		\put(-0.91, 0.2){\small$0$}
			\put(-0.91, 0.325){\small$\xi_1$}
				\put(-0.91, 0.07){\small$\xi_2$}
						\put(-0.785, 0.15){\small$I_1$}
			\put(-1.01, 0.15){\small$I_2$}
						\put(-0.41, 0.15){\small$X_1$}
			\put(-0.25, 0.15){\small$X_2$}	
		\caption{\footnotesize Set-up of the proof of non-recurrence: schematic representation the Riemann map $ \varphi\colon\mathbb{D}\to U $, together with the choice of $ \xi_1, \xi_2\in\partial \mathbb{D}$.}\label{fig-recurrencia}
	\end{figure}
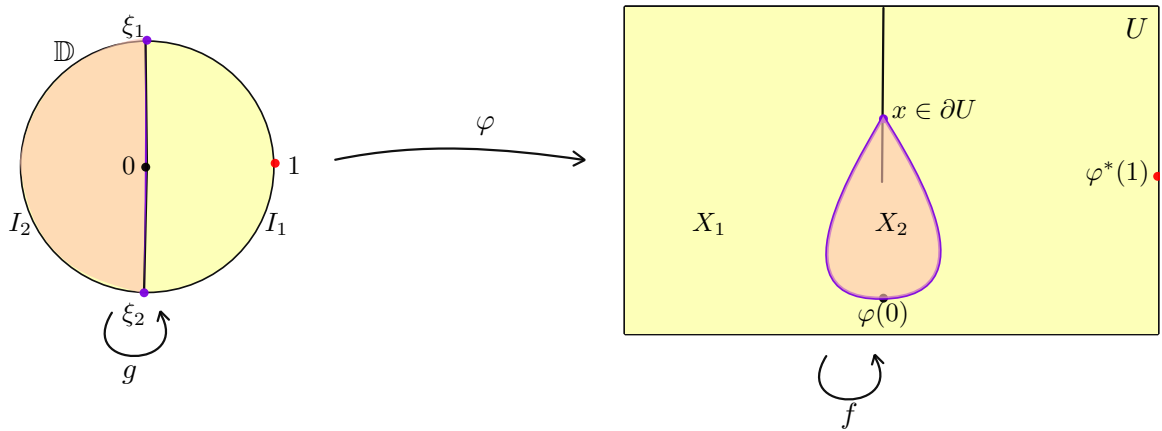

	Let us turn now to prove non-ergodicity, i.e. the existence of an invariant set of neither full nor zero measure. To do so, let \[ X\coloneqq \left\lbrace \xi\in\partial\mathbb{D}\colon \varphi^*(\xi)=a, \ \varphi^*(\xi_j)=a \textrm{ for three distinct points } \xi_j\in\partial\mathbb{D}  \right\rbrace. \] By Theorem \ref{thm-pommerenke-biaccess}, $ \lambda(X)=0 $, since it is the countable union of sets of measure zero. Moreover, since inner functions are non-singular,
	\[Z\coloneqq \partial\mathbb{D}\smallsetminus \bigcup\limits_{n\in\mathbb{Z}} (g^*)^n(X)\] is invariant under $ g^* $ and has zero $ \lambda $-measure.
	
	Since $ g^*|_{\partial\mathbb{D}} $ is non-ergodic, there exists $ A\subset Z $ such that $ \lambda (A)\in (0,1) $ and $ (g^*)^{-1}(A)=A $. Since $ f\circ\varphi^*=\varphi^*\circ g^* $ holds $ \lambda $-almost everywhere, the set $ \varphi^*(A) $ is $ f $-invariant up to a set of zero measure. If $ \omega_U(\varphi^*(A) )\in (0,1) $ we are done. 
	
	Otherwise, $ \omega_U(\varphi^*(A) )=1 $, and $ \varphi^*|_A $ is one-to-one (up to a set of zero measure). We claim that there exists $ A'\subset A $ such that $ (g^*)^{-1}(A')=A' $ and $ 0<\lambda(A')<\lambda (A) $. This would imply that $ \omega_U(\varphi^*(A') )\in (0,1) $, since $ \varphi^*|_A $ is one-to-one (up to a set of zero measure), and thus will end the proof of \ref{teo:A}.
	
	To prove the existence of $ A' $ we rely on Theorem \ref{thm-DoreingManeInner}, which claims that there exists $ h $ inner function and $ T $ Möbius transformation such that
	\[	\begin{tikzcd}
		\partial	\mathbb{D} \arrow{r}{g^*} \arrow[swap]{d}{h^*}& \partial \mathbb{D}\arrow{d}{h^*}  \\	
		\partial\mathbb{D} \arrow{r}{T} & \partial \mathbb{D}
	\end{tikzcd}\]
$ \lambda $-almost everywhere, and $ h^*|_{\partial\mathbb{D}} $ is measure-preserving. With this diagram in mind it is easy to obtain the set $ A' $ (since there exist $ T $-invariant sets of arbitrarily small measure contained in $ h^*(A) $). The proof is now complete.
\end{proof}

\section{Hyperbolic and simply parabolic Baker domains of entire functions. Proof of \ref{teo:B} and \ref{prop:C}}\label{sect-entera}

 In this section we prove \ref{teo:B}, which asserts that the non-Carathéodory set is non-empty for non-univalent Baker domains of entire functions, of hyperbolic or simply parabolic type, under the assumption that there exists a crosscut neighbourhood $ N_\xi $ of $ \xi\in\mathcal{J}(g) $ with  $ {\varphi(N_\xi)} \cap P(f) =\emptyset$. 
 
 We split the proof into three steps: first, we prove that $ \mathcal{J}(g) $ has positive Hausdorff dimension;  second, we prove topological properties of the boundaries of such Baker domains (\ref{prop:C}); and finally we prove \ref{teo:B}, which will be a consequence of the two previous results. Before starting, we include a deep study of the Carathéodory set.
 
 \subsection{The Carathéodory set}\label{subs-caratheodory}
 Recall the following definitions. 
 \begin{defi}{\bf (Crosscuts and crosscut neighbourhoods)}	A \textit{crosscut} $ C $ is an open Jordan arc $ C\subset\mathbb{D} $ such that $ \overline{C}=C\cup \left\lbrace \xi_1,\xi_2\right\rbrace  $, with $ \xi_1, \xi_2\in\partial \mathbb{D}$ and $ \xi_1\neq\xi_2 $. 
A \textit{crosscut neighbourhood} of $ \xi \in\partial\mathbb{D}$ is an open set $ N\subset\mathbb{D} $ such that $ \xi\in\partial N$, and $C\coloneqq \partial N \cap\mathbb{D} $ is a crosscut.
 \end{defi}

The previous concepts are the cornerstones of the construction of Carathéodory's compatification of a simply connected domain, and the theory of prime ends. Without diving into the subject, the Carathéodory's compactification of $ U $ endows $ \overline{U} $ with a topology such that $ \varphi\colon\mathbb{D}\to U $ extends to $ \overline{\mathbb{D}} $ as a homeomorphism, and a sequence $ \left\lbrace z_n\right\rbrace _n\subset U $, $ z_n\to \partial U $, is convergent if there exists $ \xi\in\partial\mathbb{D} $ such that for every crosscut neighbourhood $ N\subset\mathbb{D} $ of $ \xi $, there exists $ n_0 $ such that for all $ n\geq n_0 $, $ z_n\in\varphi(N) $.  For a deeper exposition on the topic, see e.g. \cite[Chap. 2.5]{Pom92}, \cite[Chap. 17]{Milnor}.

Let $ f\in\mathbb{K} $, let $ U $ be an invariant Baker domain, and let $ \varphi\colon\mathbb{D}\to U $ be a Riemann map. In the introduction,  we defined the  {\em Carathéodory set} of the Baker domain $ U $ as the set of points $ x\in\partial U $  such that, for any crosscut neighbourhood $ N\subset \mathbb{D} $ at the Denjoy-Wolff point $ p\in\partial \mathbb{D} $,  there exists $ k_0 $ such that, for all $ k\geq k_0 $, \[f^{k}(x)\in\overline{\varphi(N)}.\] In other words, the Carathéodory set is the set of points in $ \partial U $ whose orbit converges to the Denjoy-Wolff point in the Carathéodory's topology of $ \partial U $.

 Let us describe the Carathéodory set in some illustrative examples. Let us start with doubly parabolic Baker domains of entire functions. Following \cite[Thm. A]{JF23}, we have that 
 \[\partial U=\bigsqcup_{\xi\in\partial\mathbb{D}} Cl(\varphi, \xi)\cap \mathbb{C}.\] In particular, this implies that the images under $ \varphi $ of disjoint crosscut neighbourhoods in $ \mathbb{D} $ have disjoint closures in $ \mathbb{C} $.  Note that the topology of the boundary of a general simply connected domain may be more complicated and the previous property need not be satisfied. In particular, we have that the Carathéodory set of such Baler domains consists precisely of points on $ \partial U $ which belong to $ Cl(\varphi, \xi) $, for $ \xi\in\partial\mathbb{D} $ with $ (g^* )^n(\xi)\to p$. The results in \cite[Thm. F]{DM91}, \cite[Thm. C]{BFJK-Escaping} imply that in many cases the Carathéodory set has harmonic measure zero; we do next a deeper analysis.
 
 Let us start with doubly parabolic Baker domains for which the Denjoy-Wolff point $ p $ of the associated inner function $ g $ is not a singularity. In this case, $ p $ is a parabolic fixed point with two petals, and it is (weakly) repelling when restricted to $ \partial \mathbb{D} $. The only points on $ \partial\mathbb{D} $ that converge to $ p $ under iteration are its (radial) iterated preimages. Thus, the Carathéodory set is  \[\bigcup_{n\geq 0}\ \bigcup_{(g^*)^n(\xi)=p} Cl(\varphi, \xi)\cap \mathbb{C}.\] In the following example, we see that this is compatible with having a curve of escaping points in the cluster set of every $ \xi\in\partial\mathbb{D} $ (note that such points converge to $ \infty $, but the convergence does not take place through the dynamical access).
 \begin{ex}{\bf (Doubly parabolic Baker domain, finite degree, {\normalfont 	\cite[Sect. 5, 6]{BakerDominguez}, \cite[Ex. 3]{FH}, \cite{FJ23}})}\label{ex-dp-finite-deg}
 The function
 	\[f(z)=z+e^{-z}\]has a doubly parabolic Baker domain $ U_k $ of degree 2 in each strip $ S_k\coloneqq\left\lbrace (2k-1)\pi\leq \textrm{Im}z\leq (2k+1)\pi \right\rbrace  $. Due to the $ 2\pi i $-periodicity of the function, it suffices to study $ U_0 $. It is easy to see that $ \mathbb{R}\subset U_0 $ and $$L^\pm\coloneqq \left\lbrace z\in\mathbb{C}\colon \textrm{Im }z=\pm \pi i\right\rbrace  \subset\partial U_0. $$
 	
 	Fix the Riemann map $ \varphi\colon\mathbb{D}\to U_0 $ (chosen as in \cite{BakerDominguez, FJ23}). Then, the associated inner function can be computed explicitly as \[g\colon\mathbb{D}\to\mathbb{D}, \hspace{0.5cm}g(z)=\dfrac{3z^2+1}{3+z^2}.\]Note that the Denjoy-Wolff point is 1, and one can prove that $ Cl(\varphi, 1)=L^\pm\cup \left\lbrace \infty\right\rbrace  $.
 	
 	Thus, the Carathéodory set of $ U_0 $ is 
 	\[\bigcup_{n\geq 0}\ \bigcup_{g^n(\xi)=1} Cl(\varphi, \xi)\cap \mathbb{C}=\bigcup_{n\geq 0}(f^{-n}(L^\pm)).\] Note that all points in the Carathéodory set are non-accessible from $ U $.

 	Moreover, one can show that, for every $ \xi\in\partial \mathbb{D} $, $ g^n(\xi) \neq 1$ for all $ n\geq 0 $, $ Cl(\varphi, \xi) \cap\mathbb{C}$ consists of a curve of escaping points, landing at infinity from one end, and at a finite endpoint in the plane from the other end, or accumulating along itself, giving rise to an indecomposable countinua.
 	In any case, this shows the existence of plenty of escaping points which are not in the Carathéodory set. See Figure \ref{fig-BakerAnna}.
 	
 	\begin{figure}[htb!]\centering
 	\includegraphics[width=16cm]{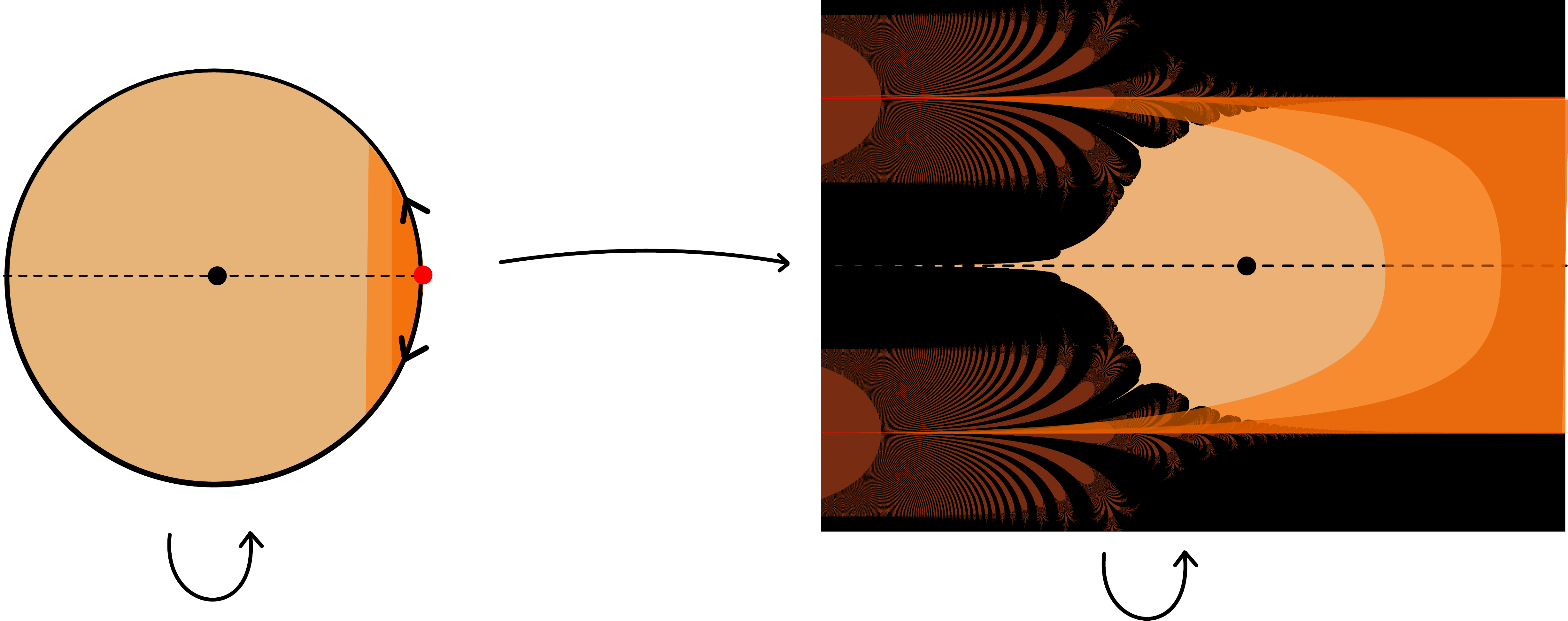}
 	\setlength{\unitlength}{16cm}
 	\put(-0.725, 0.21){\small$1$}
 	\put(-0.93, -0.02){$g(z)=\frac{3z^2+1}{3+z^2}$}
 	 	\put(-0.33, -0.02){$f(z)=z+e^{-z}$}
 	\put(-0.6, 0.245){$\varphi$}
 	\put(-0.958, 0.33){$\mathbb{D}$}
 	\put(-0.03, 0.31){$ U_0$}
 	\put(-0.21, 0.203){\small$0$}
 		\put(-0.865, 0.197){\small$0$}
 	\caption{\footnotesize Dynamical plane of $f(z)=z+e^{-z}$, with the doubly parabolic Baker domain $ U_0 $ (orange). The Riemann map $ \varphi\colon\mathbb{D}\to U_0 $ is depicted, together with the inner function. Note that 1 is the Denjoy-Wolff point, and it is repelling when restricted to $ \partial\mathbb{D} $. Crosscut neighbourhoods at the Denjoy-Wolff point are indicated, as well as their image in the dynamical plane. By definition, the Carathéodory set consists of those points on $ \partial U $ whose orbit eventually enters the image of every crosscut neighbourhood of 1. By the dynamics of $ g $, one deduces that the Carathéodory set is $ Cl(\varphi, 1)\cap \mathbb{C} $ and its iterated preimages.}\label{fig-BakerAnna}
 \end{figure}
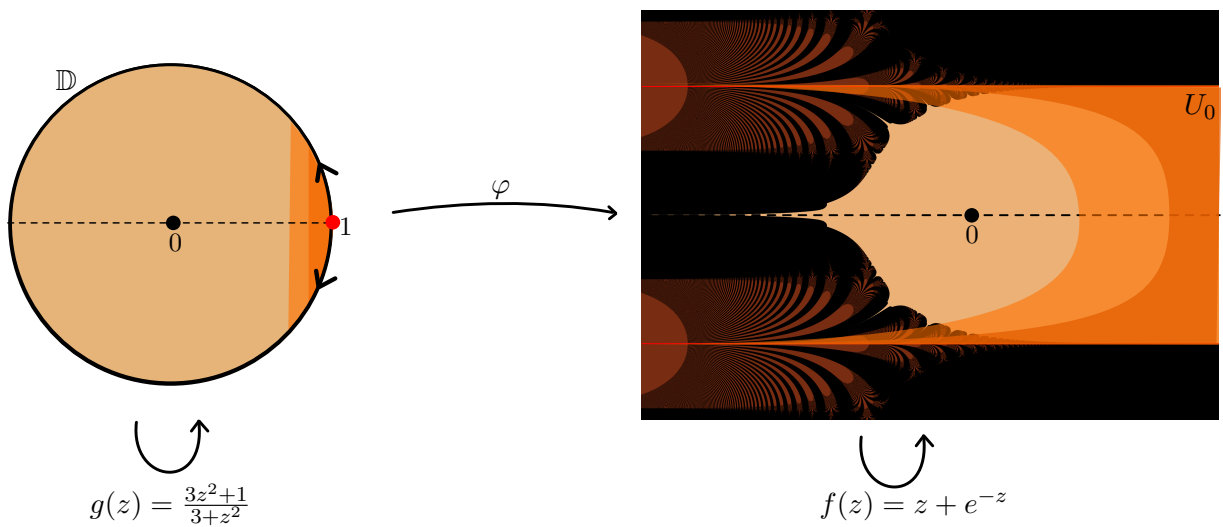
 \end{ex}
Next we deal  with doubly parabolic Baker domains for which the Denjoy-Wolff point $ p $ of the associated inner function $ g $ is  a singularity. In this case, the Carathéodory set may be larger, as shown in the following example. 
 \begin{ex}{\bf (Doubly parabolic Baker domain, infinite degree, {\normalfont  \cite{Evdoridou_Fatousweb}, \cite[Ex. 1.5]{BFJK-Escaping}})}\label{ex-dp-infinite-deg} The function
	 	\[f(z)=z+1+e^{-z},\] known as Fatou's function, has a completely invariant Baker domain $ U $, which contains a right half-plane, and $ \mathcal{J}(f)=\partial U $. It follows from \cite[Thm. D]{BFJK-Escaping} that $ \omega_U $-almost every point has a dense orbit, however we claim that the Carathéodory set is non-empty, and in fact contains accessible points. Indeed, it corresponds to the hairs whose endpoint (and thus the whole hair) escapes to $ \infty $. Such endpoints (which have harmonic measure zero) have been studied in \cite{Evdoridou_Fatousweb}.
	 	
	 	Moreover, the associated inner function can be computed explicitly as 
	 	\[g\colon\mathbb{H}\to\mathbb{H}, \hspace{0.5cm} g(z)=z- \frac{\cot z}{2},\]for a suitable Riemann map $ \varphi\colon\mathbb{H}\to U $ \cite[Thm. 1.9]{FatousAssociates}. Escaping endpoints correspond to points in $ \mathbb{R} $ which converge to the Denjoy-Wolff point under iteration of $ g $.
	 	
	 	See Figure \ref{fig-BakerFatou}.
	 	
	 	 	\begin{figure}[htb!]\centering
	 		\includegraphics[width=16cm]{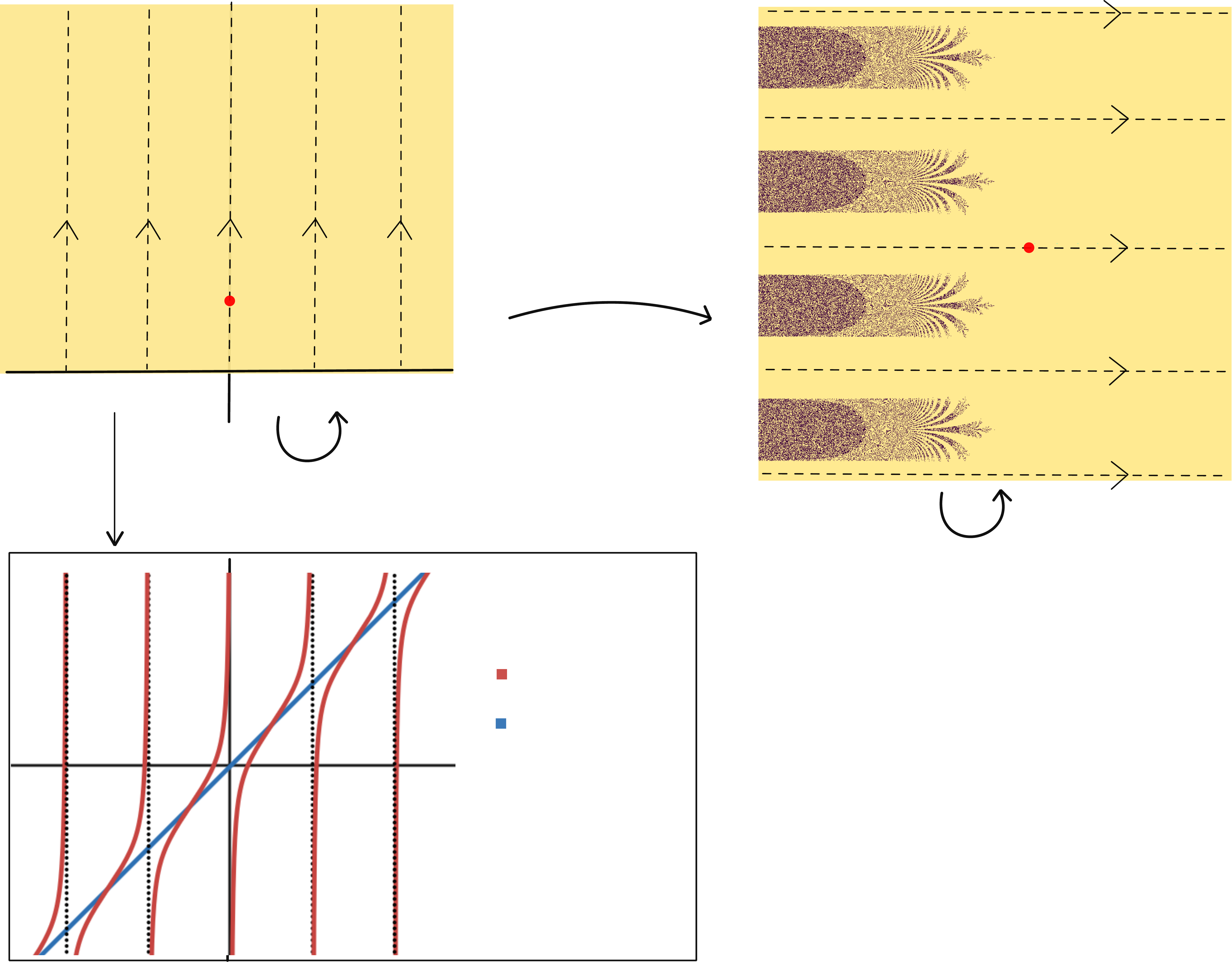}
	 		\setlength{\unitlength}{16cm}
	 		\put(-0.58, 0.23){\footnotesize$g(x)=x-\frac{\cot x}{2}$}
	 		 		\put(-0.58, 0.19){\footnotesize$y=x$}
	 		\put(-0.8, 0.39){\small $g(z)=z- \frac{\cot z}{2}$}
	 		\put(-0.31, 0.32){\small $f(z)=z+1+e^{-z}$}
	 		\put(-0.51, 0.545){$\varphi$}
	 		\put(-0.94, 0.37){\small $g|_{\mathbb{R}}$}
	 		\put(-0.02, 0.74){$ U$}
	 		 		\put(-1, 0.75){$ \mathbb{H}$}
	 		\put(-0.17, 0.59){\small$0$}
	 		\put(-0.8, 0.53){\small$i$}
	 		\caption{\footnotesize Dynamical plane of $f(z)=z+1+e^{-z}$, with the doubly parabolic Baker domain $ U $ (yellow), of infinite degree. The Riemann map $ \varphi\colon\mathbb{H}\to U $ is depicted, together with the inner function, and the graphic of the inner function restricted to the real line, $ g|_{\mathbb{R}} $. Note that $ \infty $ is the Denjoy-Wolff point,  and it is a singularity of $ g $. It is easy to see that there exists points in $ \mathbb{R} $ (which are not poles or prepoles) converging to $ \infty $ under iteration; this points correspond to escaping endpoints in the dynamical plane, and their hairs, and they form the Carathéodory set of $ U $.
	 			}\label{fig-BakerFatou}
	 	\end{figure}
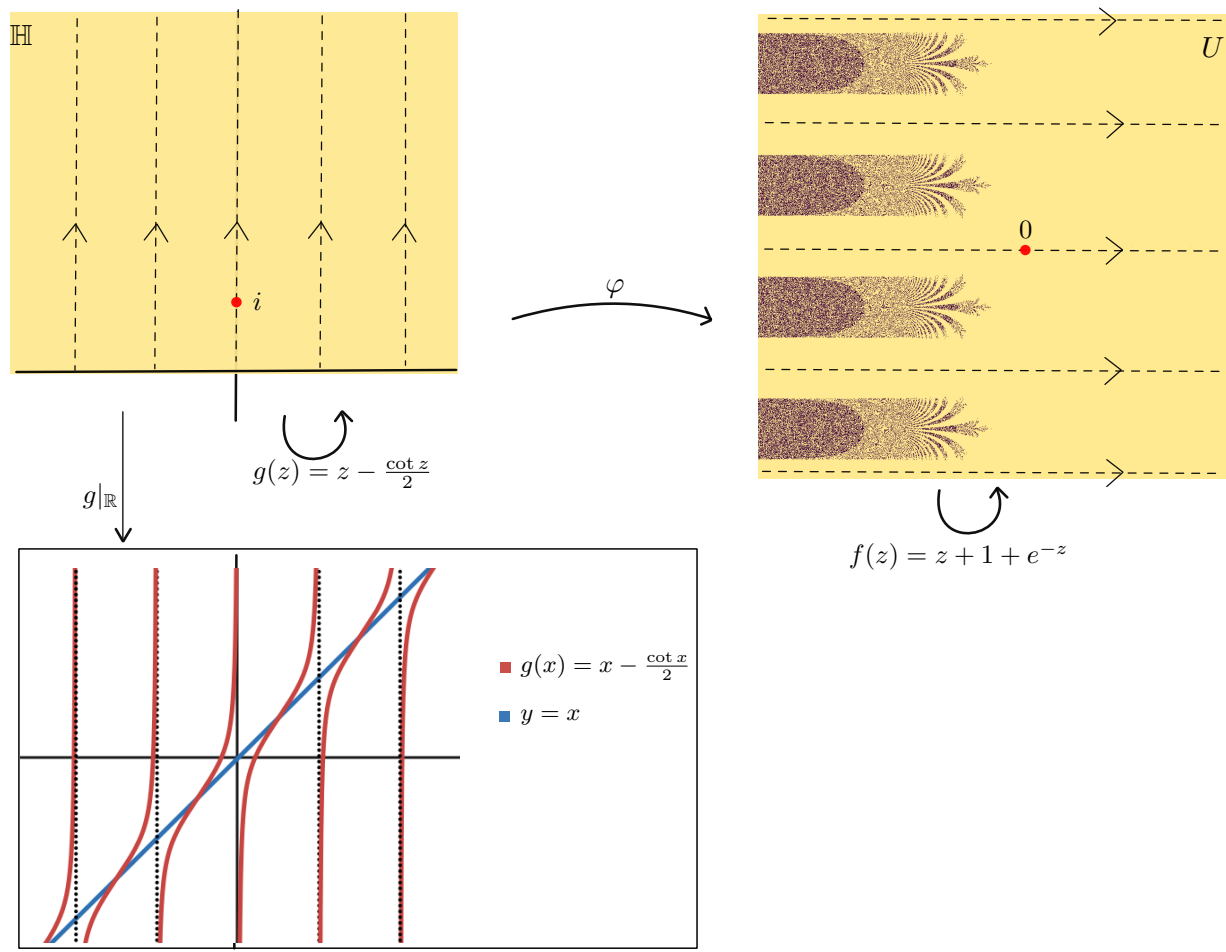
\end{ex}

For hyperbolic and simply parabolic Baker domains, the situation is fundamentally diferent. Indeed, when the Denjoy-Wolff point of $ g $ is not a singularity, it attracts points in the unit circle (from both sides if $ g $ is hyperbolic, or from one side if $ g $ is simply parabolic). Moreover, the Carathéodory set has always full measure (Thm. \ref{thm-ergodic-g}).

If we restrict ourselves to univalent Baker domains, the situation is even more tamer. The inner function (considered in the upper half-plane) is a Möbius transformation $ M\colon \mathbb{H}\to\mathbb{H} $, and can be taken to be $ z\mapsto \lambda z $, $ \lambda>0 $ (hyperbolic), or $ z\mapsto z\pm 1 $ (simply parabolic). Note that, in the first case, there is a single point in $ \mathbb{R} $ which does not converge to $ \infty $ (0, which is fixed), and none in the second case. For univalent Baker domains whose boundary is a Jordan curve, this implies that the Carathéodory set at most omits one point on $ \partial U $. We refer to the following examples.
\begin{ex}{\bf (Univalent Baker domain, simply parabolic type, {\normalfont \cite[p. 609]{Herman},
			\cite[Thm. 4]{BakerWeinreich}, \cite[Sect. 5.3]{BaranskiFagella}})}\label{ex-sp-univalent}
The function	\[f(z)=z+2\pi i \alpha +e^z, \]
	for appropriate $ \alpha\in \left[ 0,1\right] \smallsetminus\mathbb{Q} $, has a univalent Baker domain $ U $, of simply parabolic type, contained in a left half-plane. One can choose $ \alpha $ so that $ \widehat{\partial}U$ is a Jordan curve. In particular, the Carathéodory set is the whole boundary of the Baker domain.
\end{ex}
\begin{ex}{\bf (Univalent Baker domain, hyperbolic type, {\normalfont \cite{Bergweiler}, \cite[Sect. 5.1]{BaranskiFagella}})}\label{ex-hyp-univalent} 
The function	\[f(z)=2-\log2+2z-e^z\]has a univalent Baker domain $ U $, of hyperbolic type, contained in a left half-plane, and $\widehat{\partial} U $ is a Jordan curve. Then, Carathéodory set is the whole boundary of the Baker domain, except one point (the fixed point that corresponds to the repelling fixed point of the inner function under the Riemann map).

	Let us note that other examples of univalent Baker domains of hyperbolic type, such as the ones considered in \cite[Sect. 5.2]{BaranskiFagella} satisfy that the Carathéodory set is the whole boundary, since the cluster set of the repelling fixed point of the inner function is $ \left\lbrace  \infty\right\rbrace $.
\end{ex}

\begin{figure}[h]\centering
		\captionsetup[subfigure]{labelformat=empty}
	\hfill
	\begin{subfigure}[b]{0.443\textwidth}\centering
		\includegraphics[width=\textwidth]{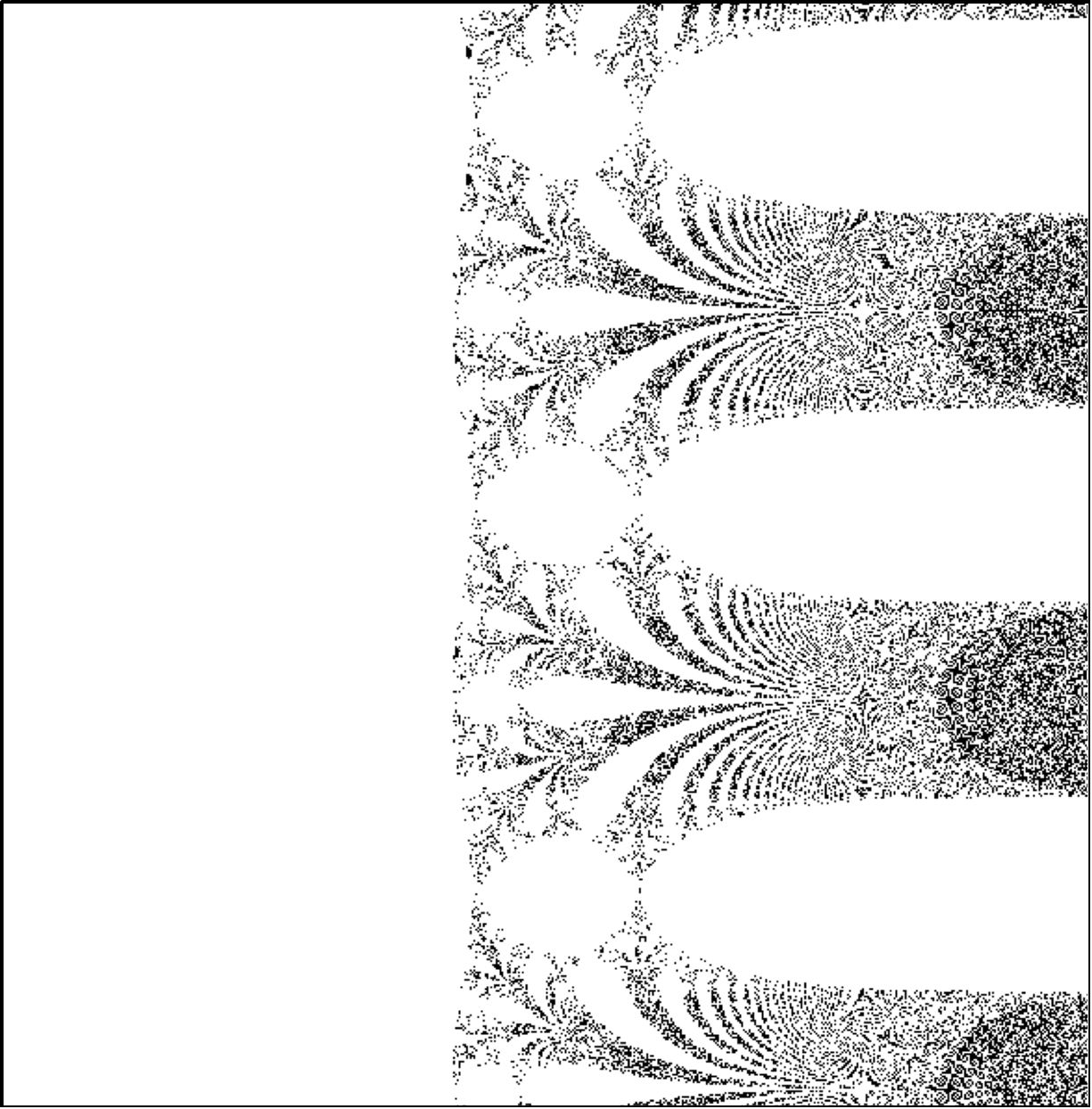}
		\caption{\footnotesize $ f(z)=2-\log2+2z-e^z $}
		
	\end{subfigure}
	\hfill
	\hfill
	\begin{subfigure}[b]{0.45\textwidth}\centering
	\includegraphics[width=\textwidth]{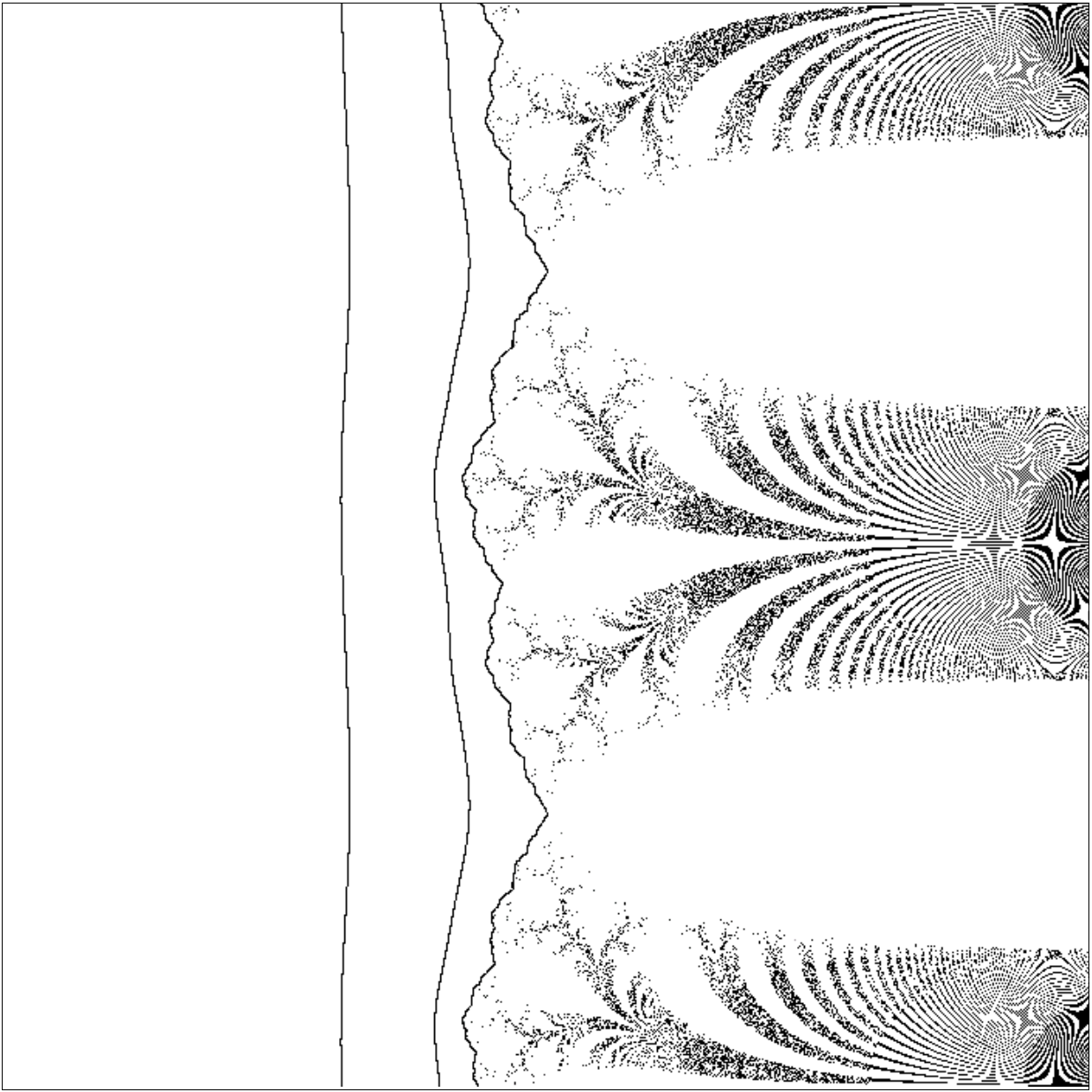}
		\caption{\footnotesize $ f(z)=z+2\pi i \alpha +e^z $}
		
	\end{subfigure}
	\hfill
	\caption{{\footnotesize Examples of univalent Baker domains of hyperbolic and simply parabolic type, respectively (in both cases, the Baker domain is the Fatou component which contains a left half-plane).}}
\end{figure}

Finally, we consider the following example of a non-univalent hyperbolic Baker domain, due to Bargmann 	\cite[Ex. 3.6]{Bargmann}, which exhibits a richer boundary dynamics. Indeed, although the Carathéodory set has full harmonic measure, periodic and bungee points are dense on $ \partial U $.
\begin{ex}{\bf (Hyperbolic Baker domain, infinite degree, {\normalfont 	\cite[Ex. 3.6]{Bargmann}})}\label{ex-hyp-bargmann}
The function
	\[f(z)=2z-3+e^z\]
	has a completely invariant Baker domain $ U $, and $ \mathcal{F}(f)=U $ and $ \mathcal{J}(f)=\partial U $. Indeed, $ \exp\circ F=F\circ\exp $, with \[F(z)=z^2e^{z-3}.\]It is easy to see that this latter function  has a completely invariant super-attracting basin $ V $ with fixed point 0, and $ \mathcal{F}(F)=V $. Thus, $ \mathcal{J}(F)=\partial V $, and by a result of Bara\'nski \cite{Baranski_trees}, $ \mathcal{J}(F) $ consists of disjoint curves (hairs) of escaping points, landing at infinity from one end and to a finite endpoint from the other end, which is the only point in the hair accessible from $ V $.  
	
	This implies that $ \mathcal{J}(f)=\partial U $, and $ \mathcal{J}(f) $ consists also of  hairs of escaping points, contained in the right half-plane. The Carathéodory set consists of those hairs whose endpoint $ x $ (and thus the whole hair) converges to infinity with $ \left| \textrm{Im}f^n(x)\right| \to\infty $. Note that the set of such endpoints have full harmonic measure.
	
	 	\begin{figure}[htb!]\centering
		\includegraphics[width=16cm]{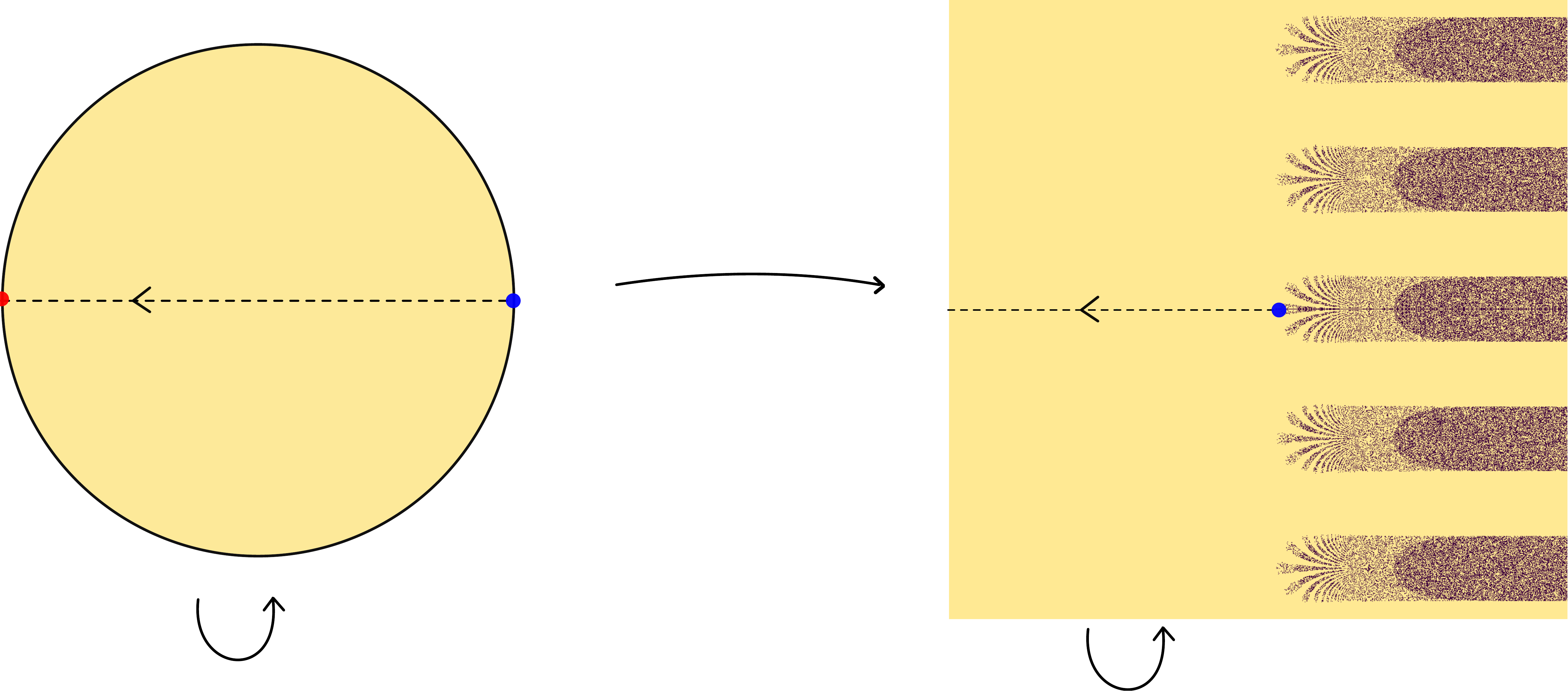}
		\setlength{\unitlength}{16cm}
		\put(-0.665, 0.24){\small$1$}
		\put(-1.03, 0.24){\small$-1$}
		\put(-0.86, -0.0){$g$}
		\put(-0.36, -0.02){\small$f(z)=2z-3+e^{z}$}
		\put(-0.53, 0.28){$\varphi$}
		\put(-0.958, 0.33){$\mathbb{D}$}
		\put(-0.39, 0.42){$ U$}
		\caption{\footnotesize Dynamical plane of $f(z)=2z-3+e^{z}$, with the hyperbolic Baker domain $ U $ (yellow) of infinite degree. The Riemann map $ \varphi\colon\mathbb{D}\to U $ is depicted, together with the inner function. Note that $ -1 $ is the Denjoy-Wolff point. The boundary $ \partial U $ is a Cantor bouquet, and the Carathéodory set consists of those endpoints whose orbit converges to $ \infty $ with $ \left| \textrm{Im}f^n(x)\right| \to\infty $ and the corresponding hairs. Such endpoints have full harmonic measure.}\label{fig-BakerBargmann}
	\end{figure}
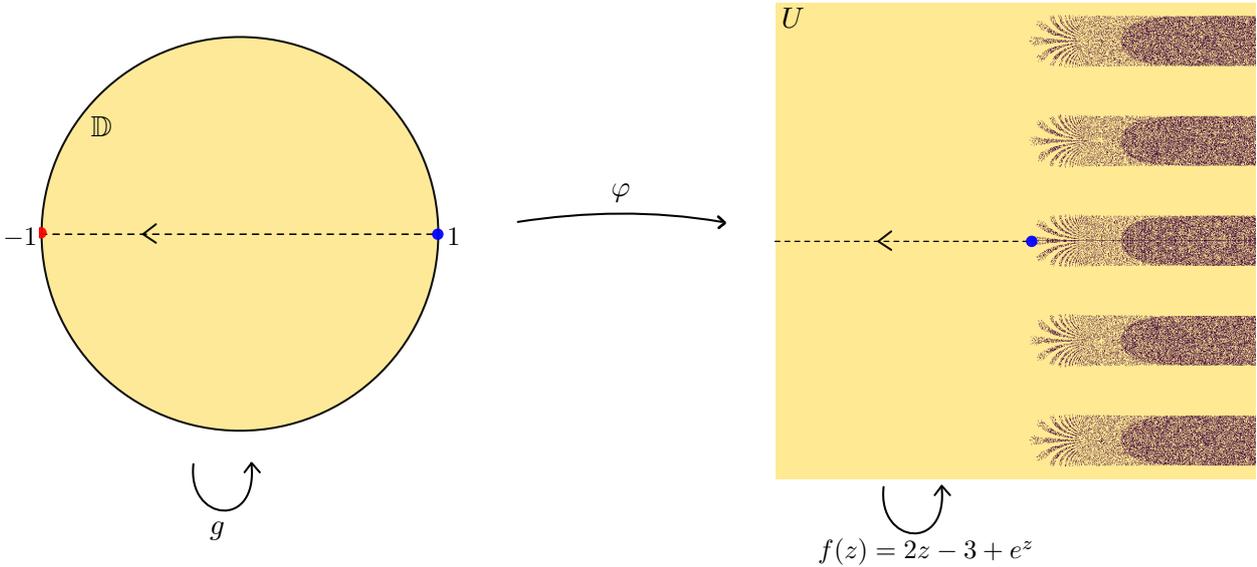
	
	Thus, observe that the Carathéodory set is dense, but non-Carathéodory set is also dense. Indeed, repelling periodic points are  dense on $ \partial U $ (and accessible from $ U $), as well as points whose orbit is dense on $ \partial U $.
	Note also the existence of escaping points which are not in the Carathéodory set. We shall remark the poor regularity of the associated inner function $ g $: its Denjoy-Wolff point  is a singularity (in fact, the only singularity of $ g $), and
	$ \mathcal{J}(g)=\partial\mathbb{D} $.
\end{ex}
\subsection{The Hausdorff dimension of $ \mathcal{J}(g) $}

We prove the following.
\begin{lemma}{\bf (Hausdorff dimension of $ \mathcal{J}(g) $)}\label{lemma-hausdorff-julia}
Let $ f\in\mathbb{K}$, and let $ U $ be a non-univalent simply connected Baker domain, of hyperbolic or simply parabolic  type. Let $ \varphi\colon\mathbb{D}\to U $ be a Riemann map, and let $ g=\varphi^{-1}\circ f\circ \varphi $ be the inner function associated  with $ f|_U $ by $ \varphi $.  
Assume there exists a crosscut neighbourhood $ N_\xi $ of $ \xi\in\mathcal{J}(g) $ such that  $ {\varphi(N_\xi)} \cap P(f) =\emptyset$.  Then, the Hausdorff dimension of $ \overline{N_\xi}\cap\mathcal{J}(g)  $ is positive.
\end{lemma}

Note that, in the case where $ g $ has finite degree (and thus it is a rational function), the fact that the  Hausdorff dimension of $ \mathcal{J}(g) $ is positive follows from a result of Garber \cite[Thm. 1]{Garber}.

Our proof follows the same ideas as in \cite[Thm. A]{Stallard}, where Stallard prove that the Julia set of a meromorphic function has always positive Hausdorff dimension. In its turn, her result relies on the following lemma \cite[Prop. 9.7]{Falconer} (see also \cite[Lemma 2.1]{Stallard}).
\begin{lemma}\label{lemma-falconer}
	Let $ \phi_1 $, $ \phi_2 $ be contractions on a closed set $ D\subset\mathbb{C} $ such that for each $ x,y\in D $, we have \[ b_i\left| x-y\right| \leq \left| \phi_i(x)-\phi_i(y)\right| \] with $ 0<b_i<1 $, $ i=1,2 $. Let $ F\subset D $ be such that \[ F=\phi_1(F)\sqcup \phi_2(F).\] Then, the Hausdorff dimension of $ F $ is greater than $ s $, where $ b^s_1+b^s_2=1 $.
\end{lemma}
\begin{proof}[Proof of Lemma \ref{lemma-hausdorff-julia}]
	Since $ {\varphi(N_\xi)} \cap P(f) =\emptyset$, according to Proposition \ref{prop-SVinner}, there exists a disk $ D $ such that $ D\cap\mathcal{J}(g)\neq\emptyset $, and all branches $ G_n $ of $ g^n $ are well-defined (and univalent) in $ \overline{D} $. Without loss of generality, let us assume that the Denjoy-Wolff point of $ g $ is not in $ D $.
	
	Recall that preimages of any point in $ \mathcal{J}(g) $ are dense in $ \mathcal{J}(g) $ (Lemma \ref{lemma-preimages-inner}), and, as an straightforward consequence of Koebe distortion theorem together with the fact that Julia sets have empty interior, given a sequence of iterated inverse branches $ \left\lbrace G_n\right\rbrace _n $, $ \textrm{diam } G_n(D)\to 0 $.  Therefore, we can choose $ \phi_1= G_{n_1}|_D $ and $ \phi_2= G_{n_2}|_D $ such that $ \phi_1 (D)\sqcup \phi_2(D)=\emptyset$, $ \phi_2 (D)\subset D$ and $ \phi_1 (D)\subset D$ (for appropriate $ n_1 $, $ n_2 $ and suitable inverse branches).
	
	Now the proof continues as in \cite[Thm. A]{Stallard}; we outline it for completeness. First, note that $ \phi_1 $, $ \phi_2 $ satisfy the hypothesis of Lemma \ref{lemma-falconer} (indeed, the fact that they are contractions follows from Schwarz lemma; the existence of $ b_i $ follows from the univalence of $ \phi_i $ in $ \overline{D} $).
	Next, define 
	\[D_{i_1\dots i_n}=\phi_{i_1}\circ\dots\circ \phi_{i_n}(\overline{D}),\]
	\[H=\bigcup\bigcap_{n=1}^\infty D_{i_1\dots i_n},\] where the union is taken over all possible sequences $ i_1\dots i_n $, $ i_j\in \left\lbrace 1,2\right\rbrace  $. Then, $ H $ is a non-empty compact set and  \[ H=\phi_1(H)\sqcup \phi_2(H).\]According to Lemma \ref{lemma-falconer}, $ H $ has positive Hausdorff dimension. Moreover, the orbit under $ g^* $ of points in $ H $ comes back to $ H\subset D $ infinitely often. Thus, points in $ H $ do not converge to the Denjoy-Wolff point, meaning that $ H\subset \mathcal{J}(g) $,  as desired.
\end{proof}

Note that, more precisely, we proved that the set of points in $ \overline{N_\xi}\cap\partial\mathbb{D}  $ which come back to $ \overline{N_\xi} $ infinitely often under iteration has positive Hausdorff dimension.
\subsection{Proof of \ref{prop:C}}
We prove the different statements separately.

	 	\vspace{0.4cm}
\noindent(a) If $ U $ is a non-univalent Baker domain, then $ \mathcal{J}(g)\subset\overline{\Theta} $ and $ \mathcal{J}(g) $ is uniformly perfect (Lemma \ref{lemma-properties-inner-function}, Thm. \ref{thm-accesses-and-julia}). Thus, $ \Theta $ is infinite. Moreover, given $ \xi_1,\xi_2\in\Theta $, 
the image under $ \varphi $ of the radial segments at $ \xi_1 $ and $ \xi_2 $, together with infinity, i.e.
\[\varphi(R(\xi_1))\cup\varphi(R(\xi_2))\cup \left\lbrace \infty\right\rbrace \] is a Jordan curve in $ \widehat{\mathbb{C}} $, say $ \gamma $, and $ \widehat{\mathbb{C}} \smallsetminus\gamma$ has two connected components, each containing points of $ \partial U $. It follows that $ \partial U $ has infinitely many connected components.
The fact that, for all $ \xi\in\partial \mathbb{D} $, $ Cl(\varphi, \xi)\cap\mathbb{C} $ is contained in either one or  two connected components, and $ \varphi^*(\xi)=\infty $ in the latter case, follows from Lemma \ref{lemma-disconnected-cluster-sets}.

It is left to prove that each component $ C $ of $ \partial U $ contains $ Cl(\varphi, \xi)\cap\mathbb{C} $ for a unique $ \xi\in\mathcal{J}(g) $, with at most countably many exceptions. Since prime ends are symmetric with at most countably many exceptions \cite[Prop. 2.21]{Pom92}, we shall restrict to $ \xi\in\partial \mathbb{D} $ such that $ Cl(\varphi, \xi)\cap\mathbb{C} $ is contained in a unique component of $ \partial U $.

Assume first that $ \overline{\Theta}=\partial\mathbb{D} $. Let $ \xi_1, \xi_2\in  \partial\mathbb{D} $. Then, there exists $ \zeta_1,\zeta_2\in\Theta $, 
the image under $ \varphi $ of the radial segments at $ \zeta_1 $ and $ \zeta_2 $, together with infinity, i.e.
\[\varphi(R(\zeta_1))\cup\varphi(R(\zeta_2))\cup \left\lbrace \infty\right\rbrace \] is a Jordan curve in $ \widehat{\mathbb{C}} $, say $ \gamma $, and $ \widehat{\mathbb{C}} \smallsetminus\gamma$ has two connected components, one containing $ Cl(\varphi, \xi_1) \cap\mathbb{C}$ and the other containing $ Cl(\varphi, \xi_2) \cap\mathbb{C}$. This already implies that $ Cl(\varphi, \xi_1) \cap\mathbb{C}$ and $ Cl(\varphi, \xi_2) \cap\mathbb{C}$ lie in different components of $ \partial U $. See Figure \ref{fig-clusertsets}.

	 	\begin{figure}[htb!]\centering
	\includegraphics[width=15cm]{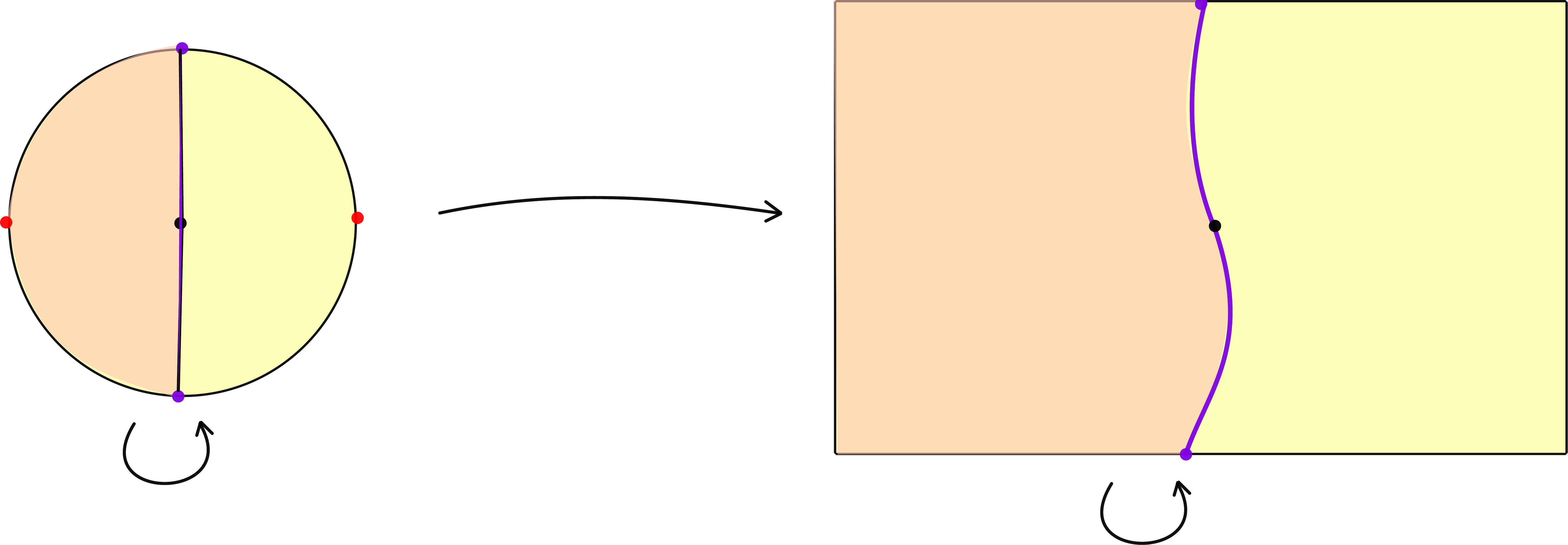}
	\setlength{\unitlength}{15cm}
	\put(-0.765, 0.2){\small$\xi_2$}
	\put(-0.91, 0.02){$g$}
	\put(-0.28, -0.025){$f$}
	\put(-0.62, 0.23){$\varphi$}
	\put(-0.97, 0.3){$\mathbb{D}$}
	\put(-0.03, 0.32){$ U$}
	\put(-0.23, 0.33){\small$\varphi^*(\zeta_1)$}
	\put(-0.235, 0.065){\small$\varphi^*(\zeta_2)$}
	\put(-0.275, 0.2){\small$\varphi(0)$}
	\put(-0.91, 0.2){\small$0$}
	\put(-0.9, 0.325){\small$\zeta_1$}
	\put(-0.9, 0.07){\small$\zeta_2$}
	\put(-1.02, 0.2){\small$\xi_1$}
	\caption{\footnotesize Set-up of the proof of \ref{prop:C}(a): schematic representation the Riemann map $ \varphi\colon\mathbb{D}\to U $, together with $ \xi_1, \xi_2\in\partial \mathbb{D} $, and the choice of $ \zeta_1, \zeta_2\in\Theta $.}\label{fig-clusertsets}
\end{figure}
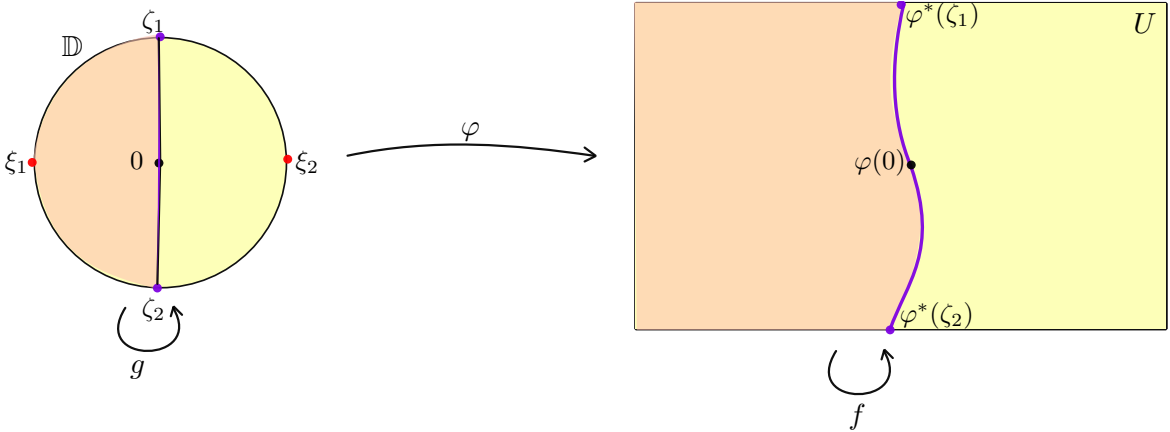

Now, assume $ \overline{\Theta}\neq\partial\mathbb{D} $. Then, $ \partial\mathbb{D}\smallsetminus \overline{\Theta}$ is open, so consists of countably many (open) circular intervals on $ \partial\mathbb{D} $. These open intervals (which are in $ \mathcal{F}(g) $), together with their endpoints (in $ \mathcal{J}(g) $), correspond to countably many components of $ \partial U $, which we disregard. The remaining points are in $  \overline{\Theta}$, and we can proceed as in the previous case. This ends the proof of (a).

	 	\vspace{0.4cm}
\noindent
(b) We rely on the following version of the Gross star theorem for iterates of a meromorphic function.
The proof can be found in \cite[Thm. 3.1]{martipete2023bounded}, inspired in the general version of the Gross star theorem given by Kaplan \cite[Thm. 3]{Kaplan}.
\begin{thm}{\bf (Gross star theorem for iterates)}\label{thm-gross-star}
	Let $ f\colon\mathbb{C}\to\widehat{\mathbb{C}} $ be a meromorphic function, and let $ z_0\in\mathbb{C} $ and $ k\geq 1 $ be such that $ w_0=f^k(z_0) $ is defined and $ (f^k)'(z_0)\neq 0 $. Let $ W\ni  w_0$ be a simply connected domain; for $ \theta\in (0,2\pi) $, denote by $ \gamma_\theta $ the hyperbolic geodesic of $ W $ starting at $ w_0 $ in the direction $ \theta $. Then, for almost every $ \theta\in (0,2\pi)  $, $ \gamma_\theta $ is an arc connecting $ w_0 $ to an endpoint $ w_\theta\in\partial W $, and the branch $ F_k $ of $ f^{-k} $ that maps $ w_0 $ to $ z_0 $ can be continued analytically along $ \gamma_\theta $ into $ w_\theta $.
\end{thm}

We have to prove that $ \mathcal{J}(g) $ is has empty interior (then, since it is uniformly perfect, it follows that it is a Cantor set, and, since $ \textrm{sing}(g)\subset \mathcal{J}(g)  $, $ \textrm{sing}(g) $ has empty interior). On the contrary, assume $ I\subset\mathcal{J}(g)$ is a (non-degenerate) circular interval.

The set \[X\coloneqq\bigcup\left\lbrace \varphi(R(\xi))\colon \xi\in\Theta\right\rbrace \cup \left\lbrace \infty\right\rbrace \] divides $ \widehat{\mathbb{C}} $ (and $ \partial U $) into infinitely many  components. Let $ J\subset\partial\mathbb{D}\smallsetminus \overline{\Theta} $ be a maximal circular interval. Then, $ Cl(\varphi, J) $ is contained in one  component of $ \partial U$, say $ C $,  and any other interval in $ \partial\mathbb{D}\smallsetminus \overline{\Theta} $ is contained in a different one.

Since preimages of points in $ \partial\mathbb{D} $ are dense in $ \mathcal{J}(g) $ there exists $ I_0\subset I $ and  $ n \geq 1$ such that $ g^n(I_0)\subset J $. Since critical points of $ f^n $ are discrete, take $ z_0=\varphi^*(\xi_0)\in\mathbb{C}$, for some $ \xi_0\in I_0$, not a critical point of $ f $. Then, $ f^n(z_0)=w_0\in C$, and let $ W \ni w_0$ be a simply connected domain, such that $ W\cap\partial U\subset C$ (note that this is possible because all other components of $ \partial U $ lie in different components of $ \widehat{\mathbb{C}}\smallsetminus X $). Then, by Theorem \ref{thm-gross-star}, the branch $ F_n $ of $ f^n $ sending $ w_0 $ to $ z_0 $ can be continued along an open set containing a continua $ K_0 \subset \mathbb{C}$ in $ C $. This implies that, for an interval $ I_1\subset I_0 $, \[\left\lbrace Cl(\varphi, \xi)\colon \xi \in I_1 \right\rbrace \subset K_0.\] This contradicts the fact that $ I_0\subset\mathcal{J}(g)\subset\overline{\Theta} $ (and thus accesses to infinity are dense in $ \left\lbrace Cl(\varphi, \xi)\colon \xi \in I \right\rbrace  $), implying that $ \mathcal{J}(g)$ is nowhere dense in $ \partial\mathbb{D} $.

Finally, if the Denjoy-Wolff point $ p\in\partial \mathbb{D} $ is not a singularity for $ g $, it follows from the local dynamics around $ p $ (which is either an attracting or a parabolic point, with one petal, for $ g $), that the Fatou set $ \mathcal{F}(g) $ is non-empty, and \[\mathcal{J}(g)=\left\lbrace \xi\in\partial\mathbb{D}\colon (g^*)^n(\xi)\not\to p\right\rbrace \cup \bigcup_{n\geq 0} (g^*)^n(p).\] In other words, the only way of converging to the Denjoy-Wolff point is either being a preimage of it, or being in the Fatou set. According to Theorem \ref{thm-ergodic-g}, the right-hand side set has zero $ \lambda $-measure. Therefore, $ \lambda (\mathcal{J}(g))=0$ and, since $ \textrm{sing}(g)\subset \mathcal{J}(g)  $, $ \lambda (\textrm{sing}(g))=0$, as desired.

	 	\vspace{0.4cm}
\noindent (c) Assume $ \mathcal{J}(g)\neq \partial\mathbb{D}$, we have to prove that periodic points are not dense in $ \partial U $. As in (b), note that the set \[X\coloneqq \bigcup\left\lbrace \varphi(R(\xi))\colon \xi\in\Theta\right\rbrace \cup \left\lbrace \infty\right\rbrace \] divides $ \widehat{\mathbb{C}} $ (and $ \partial U $) into infinitely many connected components.

Let $ I \subset \partial\mathbb{D} \smallsetminus\mathcal{J}(g)$ be an (open) circular interval, 
 which we can choose so that $ I $ and the Denjoy-Wolff point $ p\in\partial \mathbb{D} $ can be separated in $ \partial\mathbb{D} $ by two different points $ \xi_1, \xi_2\in\Theta $, i.e. $ p $ and $ I $ lie in different  circular arcs $ I_1, I_2 $  of $ \partial\mathbb{D}\smallsetminus\left\lbrace \xi_1, \xi_2\right\rbrace  $, say $ p\in I_1 $, $ I\subset I_2 $. Note that  \[\varphi(R(\xi_1))\cup\varphi(R(\xi_2))\cup \left\lbrace \infty\right\rbrace \] is a Jordan curve in $ \widehat{\mathbb{C}} $, say $ \gamma $, and $ \widehat{\mathbb{C}} \smallsetminus\gamma$ has two connected components, $ X_1 $ and $ X_2 $, with  $  Cl(\varphi, I_i)\cap\mathbb{C} \subset X_i$.
 
Since  $ g^n|_I \to p$ uniformly on compact sets, $ g^n(I)\subset I_1 $ for large $ n $, and \[Cl (\varphi, I)\cap \mathbb{C}\subset Cl(\varphi, I_2)\cap \mathbb{C}\subset X_2, \]\[ f^n(Cl(\varphi, I)\cap \mathbb{C})=Cl(\varphi, g^n(I))\cap \mathbb{C}\subset Cl(\varphi, I_1)\cap \mathbb{C}\subset X_1.\]
This implies that there are no periodic points in $ Cl(\varphi, I_2) \cap \mathbb{C}$, and thus periodic points are not dense in $ \partial U $.

	 	\vspace{0.4cm}
\noindent 
The proof of \ref{prop:C} is now complete. \hfill $ \square$

\subsection{Proof of \ref{teo:B}} Without loss of generality, assume that the Denjoy-Wolff point $ p $ of $ g $ does not belong to $ \overline{N_\xi} $. Since $ \mathcal{J}(g) $ is uniformly perfect, we can find a crosscut neighbourhood $ M_\xi $ such that $ M_\xi\subset N_\xi $, and there exists $ \xi_1, \xi_2\in \Theta $ such that $ p $ and $ M_\xi $ do not lie in the same circular arc of $ \partial\mathbb{D}\smallsetminus\left\lbrace \xi_1, \xi_2\right\rbrace  $. Therefore, for any crosscut neighbourhood $ N\subset \mathbb{D}$ of $ p $ small enough,  $$ \overline{\varphi(N)}\cap \overline{\varphi(M_\xi)}=\emptyset .$$
 
 According to Lemma \ref{lemma-hausdorff-julia}, the Hausdorff dimension of $ \overline{M_\xi}\cap\mathcal{J}(g)  $ is positive (more precisely, the set of points in $ \overline{M_\xi}\cap\partial\mathbb{D}  $ which come back to $ \overline{M_\xi} $ infinitely often under iteration has positive Hausdorff dimension). By Lemma \ref{lemma-existence-radial-limits}, there exists $ \xi\in   \overline{M_\xi}\cap\mathcal{J}(g)$ such that $ \varphi^*(\xi) \in\Omega (f)$. We can assume that the orbit of $ \xi $ comes back to $ M_\xi $ infinitely often. 
 
 We claim that $ \varphi^*(\xi)  $ does not belong to the Carathéodory set. Indeed, if $ \varphi^*(\xi)  $ belongs to the Carathéodory set, for any crosscut neighbourhood $ N\subset \mathbb{D} $ at the Denjoy-Wolff point $ p\in\partial \mathbb{D} $,  there exists $ k_0\geq0 $ such that, for all $ k\geq k_0 $, \[f^{k}(x)\in\overline{\varphi(N)}.\] Since $ \overline{\varphi(N)}\cap \overline{\varphi(M_\xi)}=\emptyset $, this
 contradicts the fact that the orbit of $ \xi $ visits $ \overline{M_\xi} $ infinitely often. The proof of \ref{teo:B} is now complete. \hfill $ \square $

\section{Boundary dynamics. Proof of \ref{teo:D}}\label{sect-boundary-dynamics}

\begin{proof}[Proof of \ref{teo:D}] Before proceeding with the proof, note that, since  there exists a crosscut neighbourhood $ N_\xi $ of $ \xi\in\mathcal{J}(g) $ such that  $ \overline{\varphi(N_\xi)} \cap P(f) =\emptyset$, Lemma \ref{lemma-hausdorff-julia} guarantees that the Hausdorff dimension of $ \mathcal{J}(g)\cap \overline{(N_\xi)}   $ is positive. In particular, there exists $ \xi \in \mathcal{J}(g)\cap \overline{(N_\xi)} $ such that $ \varphi^*(\xi)$ exists and belongs $\Omega(f) $.
	
We split the proof in several steps.
	
	 	\vspace{0.4cm}
\noindent	(1)
	{\em Construction of a region of expansion.}
			Let us consider an appropriate neighbourhood $ W $ of $ \overline{\varphi(N_\xi )} $, and prove that, with respect to the hyperbolic metric in $ W $, inverse branches are contracting, and uniformly contracting when restricted to compact sets.
		\begin{lemma}\label{lemma-hyperbolic-metric}
			Let $ W\coloneqq \mathbb{C}\smallsetminus P(f) $. Then, the following holds.
				\begin{enumerate}[label={\em(\alph*)}]
				\item  $ f\colon f^{-1}(W)\to W $ is {\em locally expanding} with respect to the hyperbolic metric $ \rho_W $, i.e. \[\rho_W(z)\leq \rho_W(f(z))\cdot \left| f'(z)\right| , \hspace{0.5cm}\textrm{for all }z\in f^{-1}(W).\]
				\item For all $ z\in W $, there exists a neighbourhood $ D_z\subset W $ of $ z $ such that all branches $ F_n $ of $ f^{-n} $ are well-defined in $ D_z $, $ F_n(D_z)\subset W $ and \[\dist_W(F_n(x), F_n(y))\leq \dist_W(x,y), \hspace{0.5cm}\textrm{for all }x,y\in D_z.\]
				\item Let $ z\in W $, and let $ D_W(z,R) $ be the hyperbolic disk centered at $ z $ of radius $ R $. Then, there exists $ C\in (0,1) $ such that, if $ F_n(z)\in D_W(z, 2R) $ and $ r\in (0,R) $, then 
				\[ F_n(D_W(z,r))\subset D_W(F_n(z), C\cdot r )\subset D_W(z, 3R).\] 
			\end{enumerate}
		\end{lemma}
	
	Note that $ \overline{\varphi(N_\xi)}\subset W $, since $ \overline{\varphi(N_\xi)}\cap P(f)=\emptyset $. Note also that the constant $ C $ does not depend on $ n $.
		\begin{proof}
			The first two items are standard in transcendental dynamics (see e.g. \cite[Prop. 5.2]{JF23}, note however that $ W $ may be disconnected). For the third item, 	note that, by the Schwarz-Pick lemma \cite[Thm. I.4.1]{CarlesonGamelin} and the triangle inequality, we have that \[ F_n(D_W(z,r))\subset D_W(F_n(z), r )\subset D_W(z, 3R).\] This last disk is relatively compact in $ W $, and hence there exists $ C\in (0,1) $ such that \[\rho_W(z)\leq C\cdot \rho_W(f(z))\cdot \left| f'(z)\right| , \hspace{0.5cm}\textrm{for all }z\in D_W(z, 3R).\] Then, by the Schwarz-Pick lemma \cite[Thm. I.4.1]{CarlesonGamelin},
			\begin{align*}
				\rho_W(z)&\leq  C\cdot \rho_W(f(z))\cdot \left| f'(z)\right| \\&\leq C\cdot  \rho_W(f^{n-1}(f(z)))\cdot\left| (f^{(n-1)})'(f(z))\right| \cdot  \left| f'(z)\right| \\&\leq C\cdot \rho_W(f^n(z))\cdot \left| (f^n)'(z)\right|  ,
			\end{align*} for all $ z\in D_W(z, 3R) $, implying the first inclusion.  The last inclusion follows from the triangle inequality. 
		\end{proof}
	
		 	\vspace{0.4cm}
	\noindent	(2) {\em Construction of accessible periodic points.} Let us start with the following lemma.
		\begin{lemma}\label{lemma-construction-periodic-point}
			Let $ \zeta\in \mathcal{J}(g) $ be such that $ x=\varphi^*(\zeta)\in \mathbb{C}$ and there exists a crosscut neighbourhood $ N_\zeta $ such that $ \overline{\varphi(N_\zeta )} \cap P(f)=\emptyset$. Then, for all $ r>0 $, there exists a periodic point $ p $ in $ D_W(x,r) $, which is accessible from $ U $.
			
			\noindent Moreover, given a finite collection of crosscut neighbourhoods $ \left\lbrace N_{\xi_k}\right\rbrace _{k=1}^n $, with $ \xi_k\in\mathcal{J}(g) $, $ p $ can be taken with $  f^{n_k}(p)\in\overline{\varphi(N_{\xi_k} )}$, for appropriate $ n_k\in\mathbb{N} $.
		\end{lemma}
	Note that, in particular, $ p $ can be taken of arbitrarily large period.
\begin{proof}
	Let us start by proving the following claim.
	\begin{claim}\label{claim}
		There exists $ \rho>0 $, a closed non-degenerate circular interval $ I$ and a subset $ K\subset I $ such that
		\begin{enumerate}
			\item $ \overline{K}=I$;
			\item for all $ \eta\in K$, $ \Delta_\rho (\zeta)\subset D_W(x, r/4) $;
				\item for all $ \eta_1\in I $ and $ \rho_1>0 $, there exists $ \eta_2 \in K$ so that $ R_\rho(\eta_1)\cap \Delta_\rho (\eta_2)\neq\emptyset $;
			\item for all $ \eta\in I $, iterated inverse branches are well-defined in $ D(\eta, \rho) $ and \[ G_n(R_\rho(\eta))\subset \Delta_\rho (G_n(\eta)) \]
			where $ R_\rho $ and $ \Delta_\rho $ denote the  radial segment and Stolz angle, respectively (for some opening $ \alpha\in (0, \pi/2) $ which is fixed throughout the whole proof of \ref{teo:D}).
		\end{enumerate}  
	\end{claim}
	\begin{proof}
		Note that $ \omega_U(D_W(x,r/4))>0 $, so $ X\coloneqq (\varphi^*)^{-1}(D(x,r)) $ has positive Lebesgue measure. Without loss of generality, we can assume $ X $ is contained in a circular interval for which condition (4) is satisfied, for some $ \rho>0 $. Now, write\[
		X^n\coloneqq \left\lbrace \xi\in X\colon \Delta_{\rho/n}(\xi)\subset D_W(x,r/4)\right\rbrace.
		\]
		Since radial and angular limits coincide, we have \[X=\bigcup_{n\geq 0} X^n,\] implying that $ \lambda (X^n)>0 $ for some $ n\geq 0 $. Let $ K $ be the set of Lebesgue density points for $ X^n $. Note that $ \lambda (K)=\lambda (X^n)>0 $. Without loss of generality, we can assume that $ \overline{K} $ is connected (and hence it is a non-degenerate closed circular interval). Denote this closed circular interval by $ I $.
		
		Replacing $ \rho $ by $ \rho/n $, we constructed a circular interval $ I $ and a subset of it, $ K\subset I $, which satisfy conditions (1), (2) and (4). Then, condition (3) follows straightforward from the fact that $ K $ is dense in $ I $ (replacing $ I $ by a smaller subinterval if needed). This ends the proof of the claim.
	\end{proof}

	Without loss of generality, we can assume $ \zeta\in I $ (otherwise, replace $ \zeta $ by a point in $ I $, and prove the lemma for this new point; this already implies the theorem for the original $ \zeta $).

Now, let $ C $ be the constant of hyperbolic contraction given by Lemma \ref{lemma-hyperbolic-metric}(c), and let $ n_0 $ be such that $ C^{n_0}<1/2 $. Since iterated preimages of any point in $ \mathcal{J}(g) $ are dense in $ \mathcal{J}(g) $ (Lemma \ref{lemma-preimages-inner}), we can find an inverse branch $ G_n $ of $ g^n $, well-defined in $\bigcup_{\eta\in I} D(\eta, \rho)$, such that, if we denote by $ G_j $, $ j=1,\dots, n $, the inverse branch of $ g^j $ sending $ \zeta $ to $ g^{n-j}(G_n(\zeta)) $, it holds\begin{enumerate}
	\item $ \#\left\lbrace j=1,\dots, n\colon G_j(\zeta)\in I\right\rbrace \geq n_0 $;
	\item $ G_n(\zeta) \in I$;
	\item if we denote by $ I_{\xi_k} $ the circular arc bounded by $ N_{\xi_k} $, there exists $ j\in \left\lbrace 1,\dots, n\right\rbrace  $ with $ G_j(\zeta)\in I_{\xi_k} $.
\end{enumerate}

Now, let $ F_j $ be the inverse branch of $ f^j $ corresponding to $ G_j $ (i.e. the one such that $ \varphi\circ G_j=F_j\circ \varphi $ in $ \Delta_\rho (\zeta) $), which is well-defined in $ D_W(x,r) $. We claim that \[ F_n(D_W(x,r))\subset D_W(F_n(x), C^{n_0}\cdot r)\subset D_W(F_n(x), r/2).\] According to Lemma \ref{lemma-hyperbolic-metric}, it is enough to see that \[\#\left\lbrace j=1,\dots, n\colon F_j(x)\in D_W(x,2r)\right\rbrace \geq n_0.\] Indeed, for all $ j\in \left\lbrace 1,\dots, n\right\rbrace  $, we have that \[F_j(D_W(x,r))\subset D_W(F_j(x),r).\] Moreover, for all $  j\in \left\lbrace 1,\dots, n\right\rbrace   $  such that $ G_j(\zeta)\in I $ (what happens at least $ n_0 $ times), there exists $ \eta\in K $ such that  \[  \emptyset\neq G_j (R_\rho (\zeta))\cap \Delta_\rho (\eta)\ni z. \] Thus,  $ w=g^j(z)\in \Delta_\rho (\zeta), G_j(w)=z\in \Delta_\rho (\eta) $, so $$ \varphi(w), \varphi(G_j(w))=F_j(\varphi(w))\in D_W(x, r/4) .$$ Therefore, by the triangle inequality, \[\dist_W(F_j(x), x)\leq \dist_W(F_j(x), F_j(\varphi(w)))+\dist_W(F_j(\varphi(w)), x)\leq 2r,\] as desired.

Next we claim that \[\overline{F_n(D_W(x,r))}\subset D_W(x,r).\] Indeed, recall that $ G_n $ has been chosen so that $ G_n(\zeta)\in I $. Therefore, there exists $ \eta\in K $ such that  \[  \emptyset\neq G_n( R_\rho (\zeta))\cap \Delta_\rho (\eta)\ni z, \] and, as above, $ w=g^n(z)\in \Delta_\rho (\zeta),G_n(w)=z\in \Delta_\rho (\eta) $, so $$ \varphi(w), \varphi(G_n(w))=F_n(\varphi(w))\in D_W(x, r/4) .$$
Note that $ \Delta_\rho (\eta)\cap \Delta_\rho (\zeta)\neq \emptyset $.

 Since we just proved that $ F_n (D_W(x,r)) \subset D_W(F_n(x), r/2)$, we have \[\dist_W(F_n(x), x)\leq \dist_W(F_n(x), F_n(\varphi(w)))+\dist_W(F_n(\varphi(w)), x)\leq \frac{r}{2}+\frac{r}{4}=\frac{3r}{4}.\] Therefore, $ F_n (D_W(x,r)) \subset D_W(x, 3r/4)$, implying the claim.

See Figure \ref{fig-puntperiodic}.

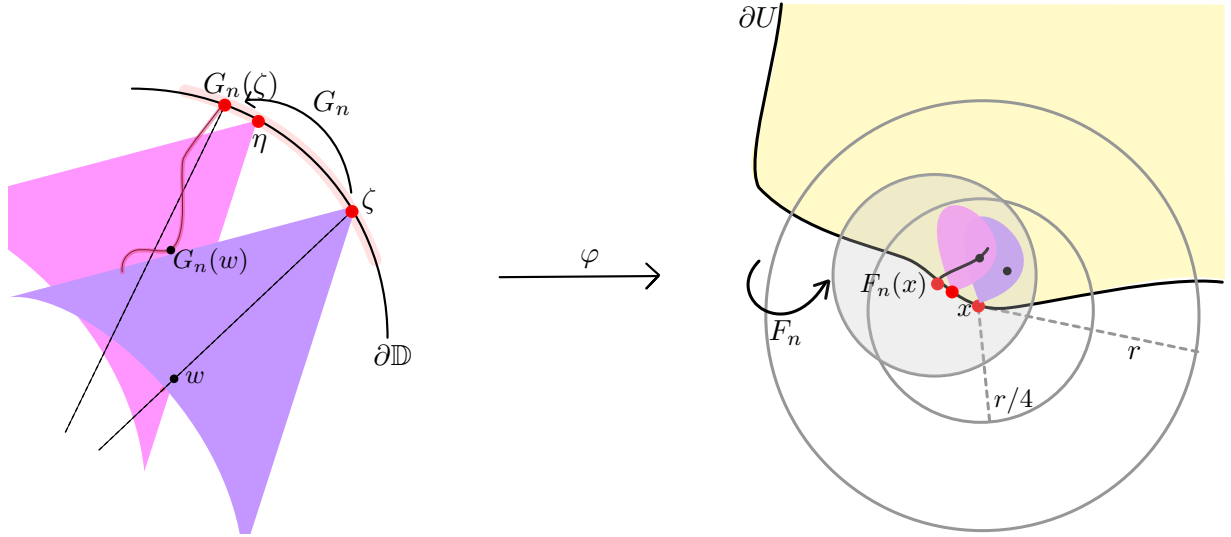
\begin{figure}[htb!]\centering
	\includegraphics[width=16cm]{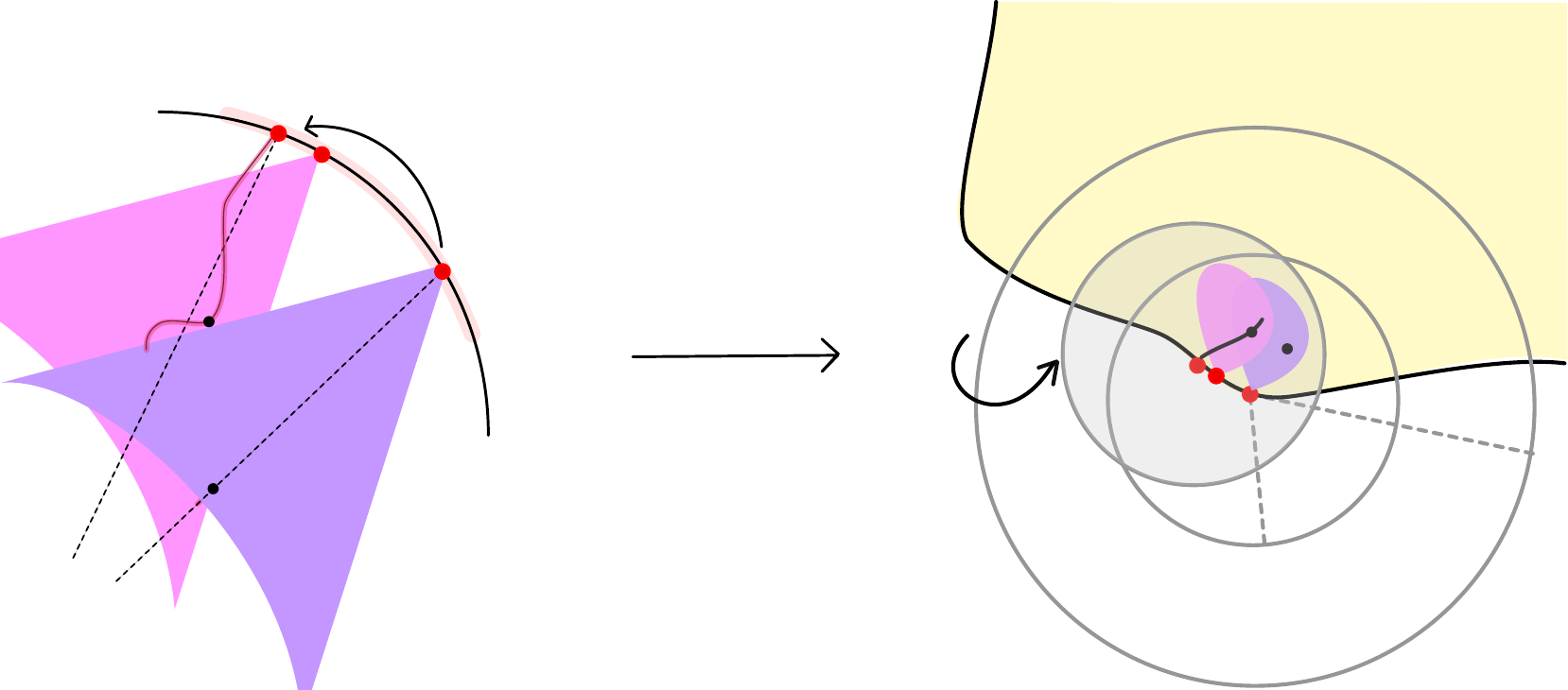}
	\setlength{\unitlength}{16cm}
	\setlength{\unitlength}{16cm}
	\put(-0.71, 0.27){$\zeta$}
	\put(-0.75, 0.35){$G_n$}
	\put(-0.375, 0.16){$F_n$}
	\put(-0.53, 0.225){$\varphi$}
	\put(-0.7, 0.14){$\partial\mathbb{D}$}
	\put(-0.4, 0.42){$\partial U$}
	\put(-0.84, 0.365){$G_n(\zeta)$}
		\put(-0.8, 0.32){$\eta$}
	\put(-0.22, 0.18){\small$x$}
	\put(-0.3, 0.2){\small$F_n(x)$}
		\put(-0.855, 0.125){\small$w$}
	\put(-0.865, 0.22){\small$G_n(w)$}
	\put(-0.08, 0.145){$r$}
	\put(-0.19, 0.105){\small $r/4$}
	\caption{\footnotesize Set-up of the proof of the existence of a periodic point in $ D(x,r) $, for $ x=\varphi^*(\zeta) \in\mathbb{C}$, $ \zeta\in\mathcal{J}(g) $. }\label{fig-puntperiodic}
\end{figure}

Finally, by the Brouwer fixed point theorem, $ F_n $ has a fixed point $ p\in D_W(x,r) $. We claim that $ p $ is accessible from $ U $, and thus $ p\in\partial U $. Indeed, let $ \gamma $ be a curve connecting $ w $ and $ G_n(w) $ in $ \Delta_\rho(\eta )\cup \Delta_\rho(\zeta )$, which exists because $ \Delta_\rho (\eta)\cap \Delta_\rho (\zeta)\neq \emptyset $.  Then, $ \varphi (\gamma) $ joins $ \varphi (w) $ and $ F_n(\varphi(w)) $ in $ U\cap D_W(x,r) $. The curve \[\Gamma=\bigcup_{m\geq 0} F^m_n (\varphi (\gamma))\] lands at $ p $, ending the proof of the lemma.
\end{proof}

	 	\vspace{0.4cm}
\noindent	(3) {\em Distribution of periodic points.}
		Statement (a) follows straightforward from Lemma \ref{lemma-construction-periodic-point}, taking into account that the preimages of $ \xi\in\partial \mathbb{D} $ are dense in $ \mathcal{J}(g) $, and, if $ \zeta= G_n(\xi) $, then $ \varphi^*(\zeta) $ exists and it belongs to $ W $. Indeed, this follows from the commutative relation $ \varphi\circ G_n=F_n\circ \varphi $ in any sufficiently small Stolz angle at $ \xi $. Then, for the radial segment $ R(\xi) $, we have $$ F_n(\varphi(R(\xi)))=\varphi(G_n(R(\xi))) .$$ By Lindelöf Theorem \cite[Thm. I.2.2]{CarlesonGamelin}, this already implies that $ \varphi^*(\zeta) $ exists (and belongs to $ W $ by the backward invariance of $ W $).   Therefore, one can apply Lemma \ref{lemma-construction-periodic-point} to $ \zeta= G_n(\xi) $ and get the desired result.
		
		If $ \mathcal{J}(g)=\partial \mathbb{D} $ and $ \omega_U(P(f))=0 $, note that one can apply Lemma \ref{lemma-construction-periodic-point} to $ \lambda $-almost every point on $ \partial\mathbb{D} $.
		
			 	\vspace{0.4cm}
		\noindent	(4) {\em Accessible points with prescribed orbit.}
		We find the desired point inductively. Indeed, let $ x_0=\varphi^*(\zeta_0) \in W$ as in Lemma \ref{lemma-construction-periodic-point}, there exists an inverse branch $ F_{n_1} $, well-defined in $ D_W(x_0, r_0) $ such that \begin{enumerate}
			\item $ G_{n_1} (\zeta_0)=\zeta_{n_1}$, $ F_{n_1}(x_0)=\varphi^*(\zeta_{n_1})=x_{n_1}\in W\cap\partial U $;
			\item $F_{n_1}(D(x_0,r_0))\subset D_W(x_{n_1}, r_1) \subset D_W(x_0, r_0)$, $ r_1=r_0/2 $;
			\item at least one of the sets $ \left\lbrace D_W(x_0,r_0), F_1(D_W(x_0,r_0)), \dots, F_{n_1}(D_W(x_0,r_0)) \right\rbrace $ intersects $ \overline{\varphi(N_{\xi_1}) }$.
		\end{enumerate}

		Since $ x_{n_1} =\varphi^*(\zeta_{n_1})\in W\cap\partial U$, we can apply again Lemma \ref{lemma-construction-periodic-point} to find an inverse branch $ F_{n_2} $, well-defined in $ D_W(x_{n_1}, r_1) $,  satisfying the analogous properties. Therefore, proceeding inductively, we get a sequence of inverse branches $ \left\lbrace F_{n_k}\right\rbrace _k $ and points $ \left\lbrace \zeta_{n_k}\right\rbrace _k\subset\partial\mathbb{D} $, $ \left\lbrace x_{n_k}\right\rbrace _k\subset \partial U $ satisfying that the inverse branch $ F_{n_k} $, well-defined in $ D_W(x_{n_{k-1}}, r_{k-1}) $ and \begin{enumerate}
			\item $ G_{n_k} (\zeta_{n_{k-1}})=\zeta_{n_k}$, $ F_{n_k}(x_{k-1})=\varphi^*(\zeta_{n_k})=x_{n_k}\in W\cap\partial U $;
			\item $F_{n_k}(D(x_{n_{k-1}},r_{k-1}))\subset D_W(x_{n_k}, r_k) \subset D_W(x_{n_{k-1}},r_{k-1})$, $ r_k=r_{k-1}/2 $;
			\item at least one of the sets $$ \left\lbrace D_W(x_{n_{k-1}},r_{k-1}), F_1(D_W(x_{n_{k-1}},r_{k-1})), \dots, F_{n_k}(D_W(x_{n_{k-1}},r_{{k-1}})) \right\rbrace $$ intersects $ \overline{(N_{\xi_{n_k}})} $.
		\end{enumerate}
	 
	 See Figure \ref{fig-bungee}.
	 	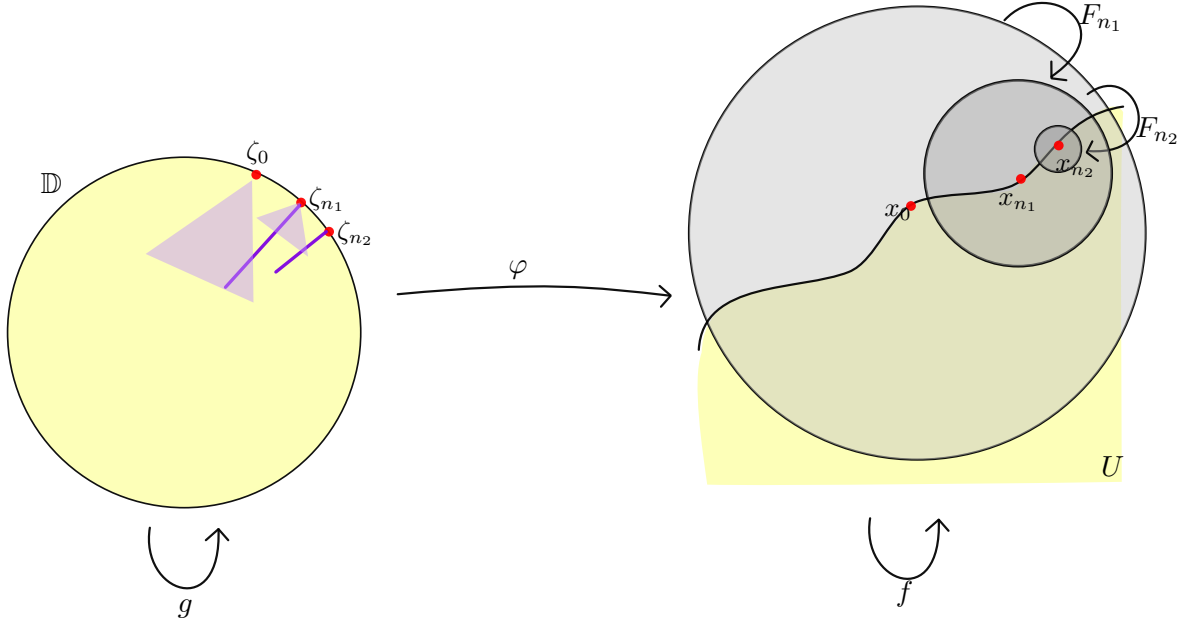
\begin{figure}[htb!]\centering
	 	\includegraphics[width=15cm]{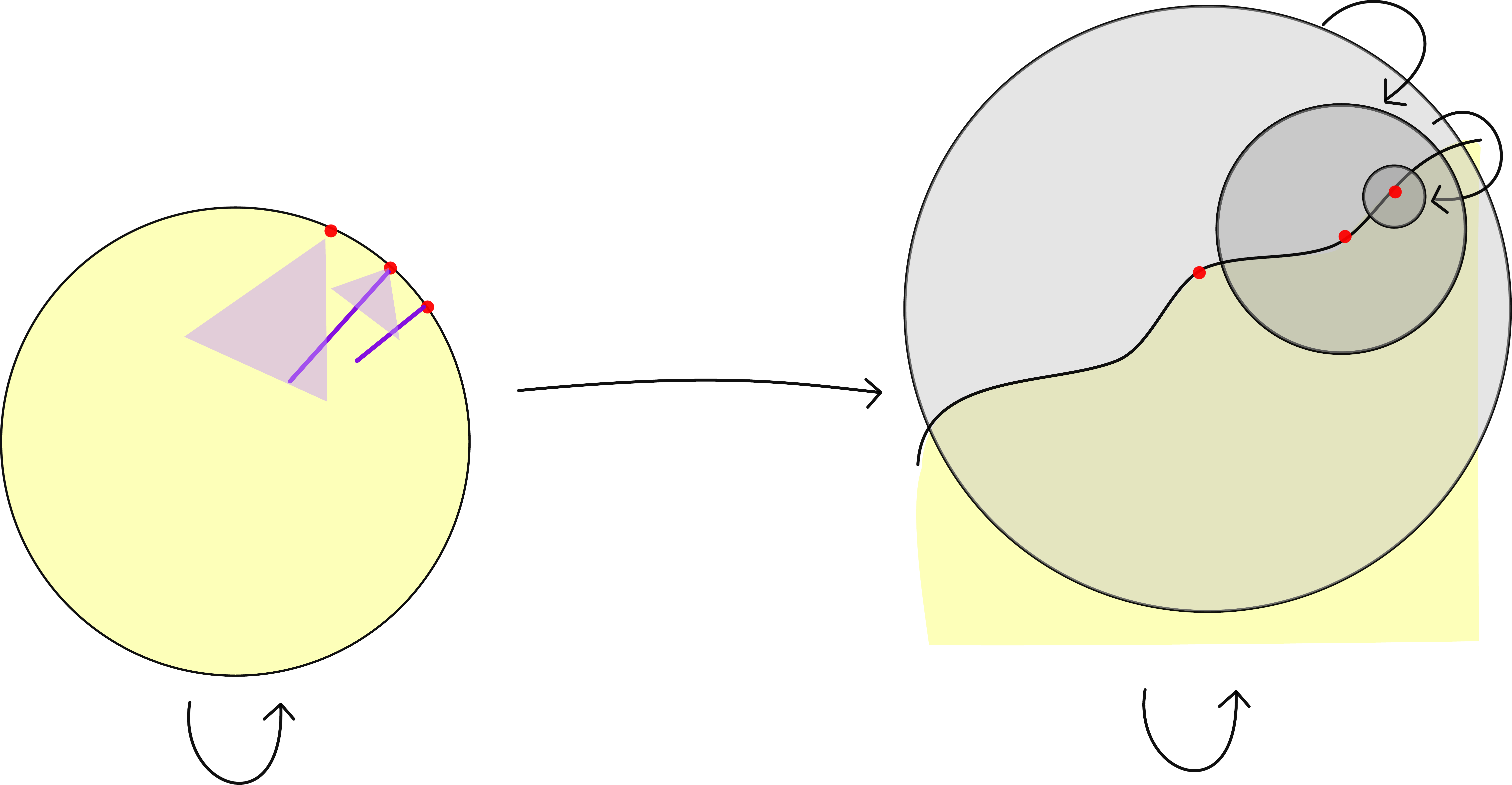}
	 	\setlength{\unitlength}{15cm}
	 	\put(-0.85, -0.02){$g$}
	 	\put(-0.22, -0.01){$f$}
	 	\put(-0.56, 0.28){$\varphi$}
	 	\put(-0.97, 0.35){$\mathbb{D}$}
	 	\put(-0.04, 0.1){$ U$}
	 	\put(-0.23, 0.33){\small$x_0$}
	 		\put(-0.06, 0.5){$F_{n_1}$}
	 	\put(-0.01, 0.4){$F_{n_2}$}
	\put(-0.13, 0.34){\small$x_{n_1}$}
		\put(-0.08, 0.37){\small$x_{n_2}$}
	 	\put(-0.735, 0.34){\small$\zeta_{n_1}$}
	 	\put(-0.71, 0.31){\small$\zeta_{n_2}$}
	 	\put(-0.79, 0.38){\small$\zeta_0$}
	 	\caption{\footnotesize Set-up of the proof of the existence of an accessible point in $ \partial U $ with a prescribed orbit.}\label{fig-bungee}
	 \end{figure}
 
		Let \[x_*\coloneqq \bigcap_k F_{n_k}(D(x_0, r_0)).\]Note that $ x_* \in\partial U$, since $ x_{n_k}\to x_* $ (and $ \partial U $ is closed).
		
		It is left to prove that $ x_* $ is accessible from $ U $. To do so, note that, by the construction in Lemma \ref{lemma-construction-periodic-point}, at step $ k $, it is satisfied that there exists a measurable set $ K_k $ and $ \eta_k \in K_k$ (chosen as in the Claim \ref{claim}) such that \[\Delta_{\rho_k}(\zeta_{n_{k-1}})\cap R_{\rho_k}(\zeta_{n_{k}})\neq \emptyset;\]
		$$ G_{n_k}(R_{\rho_k}(\zeta_{n_{k-1}}))\subset \Delta_{\rho_k} (\zeta_{n_k}); $$
			$$ G_{n_k}(R_{\rho_k}(\zeta_{n_{k-1}}))\cap  \Delta_{\rho_k} (\eta_k)\neq \emptyset; $$
		\[ \varphi(\Delta_{\rho_k}(\zeta_{n_{k-1}})), \varphi(\Delta_{\rho_k}(\eta_k))\subset D_W(x_{n_{k-1}}, r_{k-1}/4);\] for some appropriate radius $ \rho_k$. Without loss of generality, we shall assume that the sequence $ \left\lbrace \rho_k\right\rbrace _k $ is decreasing.
		
		 Thus, for all $ n\geq 0 $, take $$ z_{n_k}\in \Delta_{\rho_k}(\zeta_{n_{k}}) \cap R_{\rho_k}(\zeta_{n_{k-1}}).$$  
		 	We claim that there exists  a curve  $ \gamma_{k} $ connecting $ \varphi(z_{n_{k}}) $ to $ \varphi(z_{n_{k+1}})$ in $ D_W(x_{n_{k-1}}, r_{k-1}) \cap U$. Indeed,  since $ z_{n_k}\in \Delta_{\rho_k}(\zeta_{n_{k-1}}) $ and $ z_{n_{k+1}}\in \Delta_{\rho_{k+1}}(\zeta_{n_{k}}) $, if we prove that the set
		 	\[\Delta_{\rho_k}(\zeta_{n_{k-1}})\cup \Delta_{\rho_k}(\eta_k)\cup  G_{n_k}(\Delta_{\rho_k}(\zeta_{n_{k-1}}))\cup \Delta_{\rho_{k+1}}(\zeta_{n_{k}})\] is connected, since \[\varphi(\Delta_{\rho_k}(\zeta_{n_{k-1}})), \varphi(\Delta_{\rho_k}(\eta_k)) \subset D_W(x_{n_{k-1}}, r_{k-1}),\]
		 	\[\varphi(G_{n_k}(\Delta_{\rho_k}(\zeta_{n_{k-1}})))=F_{n_k}(\varphi(\Delta_{\rho_k}(\zeta_{n_{k-1}})))\subset F_{n_k}(D_W(x_{n_{k-1}}, r_{k-1}))\subset D_W(x_{n_{k-1}}, r_{k-1}),\]
		 	\[\varphi(\Delta_{\rho_{k+1}}(\zeta_{n_k}))\subset D(x_{n_k}, r_k)\subset D(x_{n_{k-1}, r_{k-1}}),\] then the existence of the curve $ \gamma_k $ follows straightforward.
		 	
		 	To see that $ \Delta_{\rho_k}(\zeta_{n_{k-1}})\cup \Delta_{\rho_k}(\eta_k)\cup G_n(\Delta_{\rho_k}(\zeta_{n_{k-1}}))\cup \Delta_{\rho_{k+1}}(\zeta_{n_{k}}) $ is connected, note that, on the one hand, $  \Delta_{\rho_k}(\zeta_{n_{k-1}})\cap \Delta_{\rho_k}(\eta_k)\neq\emptyset$, by the Claim \ref{claim} in Lemma \ref{lemma-construction-periodic-point}. Then, $ G_{n_k}(R_{\rho_k}(\zeta_{n_{k-1}}))\cap  \Delta_{\rho_k} (\eta_k)\neq \emptyset $, by the choice of $ \eta_k $. On the other hand, \[ R_{\rho_k}(\zeta_{n_{k-1}})\subset \Delta_{\rho_k}(\zeta_{n_{k-1}}),\]
		  and,
		 	according to Proposition \ref{prop-radial-limits}, $ G_n(R_{\rho_k}(\zeta_{n_{k-1}})) $ is a curve landing at $ \zeta_{n_{k}} $ inside $ \Delta_{\rho_k}(\zeta_{n_{k}}) $, and hence intersecting $ \Delta_{\rho_{k+1}}(\zeta_{n_{k}}) $. Hence, the intersection between $ G_n(\Delta_{\rho_k}(\zeta_{n_{k-1}})) $ and $ \Delta_{\rho_{k+1}}(\zeta_{n_{k}}) $  is non-empty. This proves the claim.

	  Finally, \[\Gamma=\bigcup_{k\geq 0} \gamma_{k}\] is an access to $ x_* $, as desired.
		 
		 	\vspace{0.4cm}
	\noindent	(5) {\em The case of infinite degree.}
		 It is left to show that, in the case of infinite degree, the set of periodic points in $ \partial U $ is unbounded and $ x_* $ can be taken bungee.
		 
		  This follows from the fact that the crosscut neighbourhood $ N_\zeta $ with $ \overline{\varphi(N_\zeta )} \cap P(f)=\emptyset$ has countably many preimages, and, for every compact subset $ K $ of $ \mathbb{C} $, only finitely many of them intersect $ K $. Hence, in the choices of crosscuts in Lemma \ref{lemma-construction-periodic-point}, we can construct a periodic point  which is outside $ D(0, n) $, leading to an unbounded sequence of periodic points. The same can be done to show that $ x_* $ can be taken to be bungee. This ends the proof of the theorem.

\end{proof}

\bibliographystyle{amsalpha}
{\footnotesize\bibliography{ref}}

\end{document}